\documentclass[11pt, a4paper]{amsart}
\usepackage{amsmath}
\usepackage{amssymb}
\usepackage{color}

\newtheorem{teo}{Theorem}[section]

\theoremstyle{definition}

\def\ep{\varepsilon}

\def\a{\alpha}
\def\R{\mathbb R}

\def\D{{\mathcal D}}

\begin{document}
\title[A heat equation with memory: large-time behavior]{A heat equation with memory: large-time behavior}

\author[Cort\'{a}zar,  Quir\'{o}s \and Wolanski]{Carmen Cort\'{a}zar,  Fernando  Quir\'{o}s, \and Noemi Wolanski}

\address{Carmen Cort\'{a}zar\hfill\break\indent
Departamento  de Matem\'{a}tica, Pontificia Universidad Cat\'{o}lica
de Chile \hfill\break\indent Santiago, Chile.} \email{{\tt
ccortaza@mat.puc.cl} }

\address{Fernando Quir\'{o}s\hfill\break\indent
Departamento  de Matem\'{a}ticas, Universidad Aut\'{o}noma de Madrid,
\hfill\break\indent 28049-Madrid, Spain,
\hfill\break\indent and Instituto de Ciencias Matem\'aticas ICMAT (CSIC-UAM-UCM-UC3M),
\hfill\break\indent 28049-Madrid, Spain.} \email{{\tt
fernando.quiros@uam.es} }

\address{Noemi Wolanski \hfill\break\indent
Departamento  de Matem\'{a}tica, FCEyN,  UBA,
\hfill\break \indent and
IMAS, CONICET, \hfill\break\indent Ciudad Universitaria, Pab. I,\hfill\break\indent
(1428) Buenos Aires, Argentina.} \email{{\tt wolanski@dm.uba.ar} }

\thanks{This project has received funding from the European Union's Horizon 2020 research and innovation programme under the Marie Sklodowska-Curie grant agreement No.\,777822. C.\,Cort\'azar supported by  FONDECYT grant 1190102 (Chile). F.\,Quir\'os supported by project MTM2017-87596-P (Spain). N.\,Wolanski supported by
CONICET PIP625, Res. 960/12, ANPCyT PICT-2012-0153, UBACYT X117 and MathAmSud 13MATH03 (Argentina).}



\date{}

\begin{abstract}
We study the large-time behavior in  all $L^p$ norms and in different space-time scales of solutions to  a heat equation with a Caputo  $\alpha$-time derivative posed in $\mathbb{R}^N$. The initial data are assumed to be integrable, and, when required, to be also in $L^p$. A main difficulty in the analysis comes from the singularity in space at the origin of the fundamental solution of the equation when~$N>1$.

The rate of decay in $L^p$ norm  in the characteristic scale, $|x|\asymp t^{\alpha/2}$, dictated by the scaling invariance of the equation,  is $t^{-\frac{\alpha N}{2}(1-\frac1p)}$. In compact sets it is $t^{-\alpha/2}$ for $N=1$, $t^{-\alpha}$ for $N\ge 3$, and $t^{-\alpha}\log t$ in the critical dimension $N=2$. In intermediate scales, going to infinity but more slowly than $t^{\alpha/2}$, we have an intermediate decay rate. In fast scales, going to infinity faster than $t^{\alpha/2}$, there is no universal rate, valid for all solutions, as we will show by means of some examples. Anyway, in such scales solutions decay faster than in the characteristic one.

When divided by the decay rate, solutions behave for large times  in the characteristic scale like $M$ times the fundamental solution, where $M$ is the integral of the initial datum. The situation is very different in compact sets, where they converge to the Newtonian potential of the initial datum  if $N\ge 3$, one of the main novelties of the paper, and to  a constant if $N=1,2$. In intermediate scales they approach a multiple of the fundamental solution of the Laplacian if $N\ge 3$, and a constant in low dimensions. The asymptotic behavior in scales that go to infinity faster than the characteristic one depends strongly on the behavior of the initial datum at infinity. We give results for certain initial data with specific decays.
\end{abstract}

\maketitle

\section{Introduction}
\label{sect-Introduction} \setcounter{equation}{0}

The aim of this paper is to study the large time behavior of solutions to the Cauchy problem
\begin{equation}
\label{eq:main}
\tag{P}
\partial_t^\alpha u=\Delta u \quad\text{in } \R^N\times\R_+,\qquad
u(\cdot,0)=u_0\quad\text{in } \R^N,
\end{equation}
where $\partial_t^\alpha$, $\alpha\in(0,1)$, is  the Caputo $\alpha$-derivative,
$$
\partial_t^\alpha u(x,t)=\frac1{\Gamma(1-\alpha)}\,\partial_t\int_0^t\frac{u(x,\tau)- u(x,0)}{(t-\tau)^{\alpha}}\, d\tau,
$$
which first appeared in~\cite{Caputo-1967}.
The initial data are always assumed to be non-negative, integrable and non-trivial, and when required, to  belong also to  $L^p(\mathbb{R}^N)$ for some $p\in (1,\infty]$.

While in normal
diffusions (described by the heat equation or more general parabolic equations), the
mean squared displacement of a particle is proportional to $t$ for $t$ large, in the time-fractional heat equation~\eqref{eq:main} it is proportional to $t^\alpha$. This
is the reason why, in the  range $\alpha\in (0,1)$, it is described  in the literature as a subdiffusion
equation. There are many real-world systems which show this kind of behavior,  for example, particle systems with sticking and trapping phenomena~\cite{Meerschaert-Sikorskii-2012,Shlesinger-Klafter-Wong-1982}, or fluids in porous media with memory~\cite{Caputo-1999}. Equation~\eqref{eq:main} and nonlinear variants of it are well suited to describe such systems. We refer to the expository review article~\cite{Metzler-Klafter-2000} for a detailed derivation of these equations from physics principles and for further applications of such models.

If both $u_0$ and its Fourier transform belong to $L^1(\mathbb{R}^N)$, problem~\eqref{eq:main} has a unique bounded classical solution given by
\begin{equation}
\label{eq:convolution}
u(\cdot,t)=Z(\cdot,t)*u_0,
\end{equation}
where $Z$ is the fundamental solution of the equation, that is, the solution to~\eqref{eq:main} having a Dirac mass as initial datum~\cite{Kochubei-1990,Eidelman-Kochubei-2004,Kemppainen-Siljander-Zacher-2017}. If $u_0$ is only known to be in $L^1(\mathbb{R}^N)$, the function $u$ in~\eqref{eq:convolution} is still well defined, but it is not in general a classical solution to~\eqref{eq:main}. However, it is a solution in a generalized sense~\cite{Gripenberg-1985,Kemppainen-Siljander-Zacher-2017}.
In this paper we will always deal with solutions of this kind, given by~\eqref{eq:convolution}, which are denoted in the literature as \emph{mild} solutions~\cite{Pruss-book,Kemppainen-Siljander-Zacher-2017}.

The main difficulty that we face here, as compared with the local case, is that $Z$ is singular in space at the origin if $N>1$. In fact, $Z\not\in L^p(B_1)$ if $p\ge N/(N-2)$.  Hence it cannot give the behavior in the $L^p$ norm in that range of values of $p$, or at least not in compact sets.

The large-time behavior will depend strongly not only on which $L^p$ norm we are dealing with, but also on the space-time scale   under consideration, with new features that are not present in the local case. For instance, solutions converge in compact sets for $N\ge 3$ to the Newtonian potential of the initial datum, one of the main and more novel results of the paper.

\medskip

\noindent\emph{Notations. } As is common in asymptotic analysis, $f\asymp g$ will mean that there are constants $\nu,\mu>0$ such that
$\nu g\le f\le \mu g$. The spatial integral of the solution, which is preserved along the evolution, and is hence equal to $\int_{\mathbb{R}^N}u_0$, will be denoted by $M $. The ball of radius $r$ centered at the origin will be denoted by $B_r$.

\subsection{The local case. } In order to have some perspective, let us start by recalling what is known for the standard local heat equation, which corresponds to $\alpha=1$. In this case the fundamental solution is given by
$$
\Gamma(x,t)=t^{-N/2}F(x/t^{1/2}),\quad\text{where } F(\xi)=(4\pi )^{-\frac N2}\textrm{e}^{-\frac{|\xi|^2}{4}},
$$
which belongs to $C^\infty(\mathbb{R}^N\times \mathbb{R}_+)$ and satisfies $\Gamma(\cdot,t)\in L^p(\mathbb{R}^N)$, $1\le p\le \infty$, for all $t>0$.
Thus, there is a smoothing effect: if the initial datum is integrable, the solution becomes immediately bounded and smooth for any positive time. Moreover, if $1\le p\le \infty$ we have the global convergence result
\begin{equation}
\label{eq:convergence.local}
\lim_{t\to\infty }t^{\frac{N}2(1-\frac1p)}\|u(\cdot,t)-M\Gamma(\cdot,t)\|_{L^p(\R^N)}=0, \quad\text{where } M=\int_{\mathbb{R}^N}u_0.
\end{equation}
As a consequence,
$$
\lim_{t\to\infty }t^{\frac{N}2(1-\frac1p)}\|u(\cdot,t)\|_{L^p(\R^N)}=M\lim_{t\to\infty }t^{\frac{N}2(1-\frac1p)}\|\Gamma(\cdot,t)\|_{L^p(\R^N)}=C\neq0,
$$
which implies the (sharp) decay rate $\|u(\cdot,t)\|_{L^p(\mathbb{R}^N)}\asymp t^{-\frac{N}2(1-\frac1p)}$ (we are assuming always that $M\neq0$).
Thus, for fixed $N$, the  decay rate increases with $p$, going from no decay for $p=1$, in agreement with the conservation of mass, to a decay $\|u(\cdot,t)\|_{L^\infty(\mathbb{R}^N)}\asymp t^{-\frac{N}2}$ for $p=\infty$. For fixed $p>1$ the decay rate increases without limit as the dimension increases.

The heat equation has a \emph{characteristic} scale, $|x|\asymp t^{1/2}$, dictated by the scaling invariance of the equation, in which the decay rates coincide with the global ones. Indeed, since
$$
\lim_{t\to\infty }t^{\frac{N}2(1-\frac1p)}\|\Gamma(\cdot,t)\|_{L^p(\{\nu\le |x|/t^{1/2}\le\mu\})}=C\neq0\quad\text{if }0<\nu<\mu<\infty,
$$
then $\|u(\cdot,t)\|_{L^p(\{\nu\le |x|/t^{1/2}\le\mu\})}\asymp t^{-\frac{N}2(1-\frac1p)}$. When properly scaled, solutions approach $M\Gamma$ in this characteristic scale.

The behavior in other scales depends on whether they are fast or slow, compared with the characteristic scale. In \emph{inner} (or slow) scales, $|x|=o(t^{1/2})$, which include both the case of compact sets~$K$ or \emph{intermediate} scales, growing slowly to infinity,
$$
|x|\asymp g(t)\quad \text{with }g(t)\to\infty,\quad g(t)=o(t^{1/2}),
$$
the global convergence result~\eqref{eq:convergence.local} implies that the (sharp) decay rate in $L^\infty$ coincides with that of the characteristic scale. Moreover,  solutions still resemble $M\Gamma$ for large times (though we only see the value of $\Gamma$ at $\xi=0$).  Results for other values of $p$ follow from the ones for $p=\infty$.

The situation for \emph{fast} scales,
$$
|x|\asymp g(t)\quad\text{with }
g(t)t^{-1/2}\to \infty,
$$
is very different
and much more involved. The global convergence result only says in this case that $\|u(\cdot,t)\|_{L^p(\{\nu\le |x|/g(t)\le\mu\})}=o(t^{-\frac{N}2(1-\frac1p)})$. Sharper results require some knowledge of the behavior of the initial datum at infinity, as the following example, due to Herraiz~\cite{Herraiz-1999}, shows.

\noindent\emph{Example. } Let $u_0\in L^1(\mathbb{R}^N)$ be such that  $|x|^\beta u_0(x)\to A$ as $|x|\to\infty$ for some constants  $A>0$ and $\beta>N$. Then,
$$
\begin{array}{lll}
u(x,t)=M  \Gamma(x,t)(1+o(1)) &\text{as }t\to\infty\text{ uniformly on }\{|x|\le \mu t^{\frac12}(\log t)^{\frac12}\},\ & 0<\mu<\sqrt{2(\beta-N)},\\[8pt]
|x|^\beta u(x,t)=A(1+o(1))\ &\text{as }t\to\infty\text{ uniformly on }\{|x|\ge \nu t^{\frac12}(\log t)^{\frac12}\},\ & \nu>\sqrt{2(\beta-N)}.
\end{array}
$$
We observe in this example an important difference between \emph{moderately fast} scales, for which the behavior is still given by $M\Gamma$, and \emph{very fast} scales,  in which the solution stays equal to the initial datum in first approximation. Notice that in moderately fast scales the decay rate in $L^\infty$-norm depends only on the scale, while  for very fast scales it also depends on the precise decay at infinity of the initial datum.  Which scales are on each side, moderately fast or very fast, also depends on the decay of $u_0$. In fact, the transition between one type of behavior to the other takes place in this example precisely at the scale for which $M\Gamma\asymp u_0$.

\subsection{The fundamental solution for the fractional problem. }
Let $\widehat{Z}=\widehat{Z}(\omega,t)$ denote the Fourier transform of the fundamental solution $Z$ of problem~\eqref{eq:main} in the $x$ variable. Then,
$$
\partial_t^\alpha \widehat{Z}(\omega,t)=-|\omega|^2\widehat{Z}(\omega,t),\qquad \widehat{Z}(\omega,0)=1.
$$
The solution to this ordinary fractional differential equation is
$$
\widehat Z(\omega,t)=\mathbb{E}_\alpha(-|\omega|^2 t^\alpha),
$$
where $\mathbb{E}_\alpha$ is the Mittag-Leffler function of order $\alpha$,
$$
\mathbb{E}_\alpha(s)=\sum_{k=0}^\infty\frac{s^k}{\Gamma(1+k\alpha)}.
$$
Inverting the Fourier transform, we finally obtain that $Z$ has a self-similar form,
$$
Z(x,t)=t^{-\alpha N/2}F(\xi),\quad \xi=x t^{-\alpha/2},
$$
with a radially symmetric positive profile $F$ that has an explicit expression in terms of certain  Fox's $H$-functions. Special functions in that family have been thoroughly studied, and the properties for the one appearing in the description of $F$ yield the following estimates for the profile,
\begin{eqnarray}
\label{eq:global.bounds.F}
	F(\xi)\le C
\begin{cases}
E_N(\xi),&N\ge3,\\
1+|E_2(\xi)|,&N=2,\\
1,&N=1,
\end{cases}\qquad &\xi\neq0,\\[8pt]
\label{eq:bounds.F.origin}
	F(\xi)\le C \exp(-\sigma|\xi|^{2/(2-\alpha)}),\qquad &|\xi|\ge 1,
\end{eqnarray}
for some positive constants $C$ and $\sigma$,
where $E_N$ is the multiple of the fundamental solution for the Laplacian in dimension $N$ given by
$$
E_N(\xi)=\begin{cases}
|\xi|^{2-N},&N\ge3,\\
-\ln |\xi|,&N=2.
\end{cases}
$$
The above estimates are sharp. Indeed, there are positive constants $\kappa=\kappa(N)$, $\hat\kappa=\hat\kappa(N)$ such that
\begin{eqnarray}
\label{eq:constant.Nge2}
F(\xi)/E_N(\xi)\to\kappa\quad&\text{as }|\xi|\to0,\quad& N\ge2,\\[8pt]
\label{eq:constant.infty}
\exp(\sigma|\xi|^{2/(2-\alpha)})	F(\xi)\to \hat\kappa\quad&\text{as }|\xi|\to\infty,\quad &N\ge 1.
\end{eqnarray}
In particular $F$ is continuous at the origin only for $N=1$. However, $F\in C^\infty$ outside the origin in all dimensions. There are also estimates for the gradient,
$$
\begin{array}{ll}
|DF(\xi)|\le C
\begin{cases}
|\xi|^{1-N},&N\ge2,\\
1,&N=1,
\end{cases}\qquad &\xi\neq0,\\[8pt]
|DF(\xi)|\le C \exp(-\sigma|\xi|^{2/(2-\alpha)}),\qquad &|\xi|\ge 1,
\end{array}
$$
and in fact for all subsequent derivatives.
These properties, which are due to Kochubei~\cite{Kochubei-1990} (see also~\cite{Kemppainen-Siljander-Zacher-2017}),
will be fundamental for our analysis.

\noindent\emph{Remarks. } (a) An estimate for the error in the limit~\eqref{eq:constant.Nge2} when $N\ge3$ is also available, namely, there is a constant~$C>0$ such that
\begin{equation}
\label{eq:estimate.error.behaviour.profile.N=3}
\big||\xi|^{N-2}F(\xi)-\kappa\big|\le C|\xi| \quad\text{if }N\ge3.
\end{equation}

\noindent (b) As a corollary of the estimates for the gradient and the Mean Value Theorem, for every $\nu>0$ there is a constant $C_\nu>0$ such that  if $|\xi-\tau\eta|\ge\nu>0$ for all $\tau\in(0,1)$, then
\begin{equation}
\label{eq:estimate.MVT}
|F(\xi-\eta)-F(\xi)|\le C_\nu \exp(-\sigma|\xi-s\eta|^{2/(2-\alpha)})|\eta|\quad \text{for some } s\in(0,1),
\end{equation}
an estimate that will be useful in future computations.

\medskip

As a first consequence of the estimates for $F$ we observe that
$Z(\cdot,t)\in L^p(\mathbb{R}^N)$  for $t>0$ if and only if $p$ is \emph{subcritical},
\begin{equation}
\label{eq:subcritical.range}
\tag{S}
p\in [1,\infty]\quad\text{if }N=1,\qquad p\in[1,p_c)\quad\text{if }N\ge 2, \quad \text{where }
p_c=\begin{cases}
N/(N-2)&\text{if }N\ge 3,\\
\infty&\text{if }N=2.
\end{cases}
\end{equation}

In the critical case $p=p_c$, $N\ge3$, $Z(\cdot,t)$ belongs to the Marcinkiewicz space
$M^{p_c}(\mathbb{R}^N)$. Marcinkiewicz spaces are $L^p$ weak spaces defined through the norm
\[
\|g\|_{M^p(\mathbb{R}^N)}=\sup\Big\{\lambda\big|\{x\in\R^n:|g(x)|>\lambda\}\big|^{1/p},\ \lambda>0\Big\}, \quad p\in [1,\infty).
\]
That is, $\|g\|_{M^p(\mathbb{R}^N)}$ is the smallest constant $C$ such
\[
\big|\{x\in\R^n:|g(x)|>\lambda\}\big|\le \Big(\frac C\lambda\Big)^p\quad \text{for all }\lambda>0.
\]
Thus, $M^p(\mathbb{R}^N)$ coincides with the Lorentz space $L^{p,\infty}(\mathbb{R}^N)$.

\subsection{Decay rates.  } If $p$ is subcritical, that is, if~\eqref{eq:subcritical.range} holds,   there is an $L^1$--$L^p$ smoothing effect: if the initial datum is in $L^1(\mathbb{R}^N)$, then $u(\cdot,t)\in L^1(\mathbb{R}^N)\cap L^p(\mathbb{R}^N)$ for every $t>0$. In that range we have the decay rate,~$\|u(\cdot,t)\|_{L^p(\mathbb{R}^N)}\asymp t^{-\frac{\alpha N}2(1-\frac1p)}$; see~\cite{Kemppainen-Siljander-Vergara-Zacher-2016}.

If $p>N/(N-2)$, $N\ge 3$, there is no smoothing effect, and in order to have $u(\cdot, t)\in L^p(\mathbb{R}^N)$ we have to require $u_0\in L^p(\mathbb{R}^N)$. In that case, if $p$ is moreover finite, Kemppainen \emph{et al} proved in~\cite{Kemppainen-Siljander-Vergara-Zacher-2016} that~$\|u(\cdot,t)\|_{L^p(\mathbb{R}^N)}\asymp t^{-\alpha}$.
We will complete the results in that paper showing that this decay rate is also valid for $p=\infty$, $N\ge 3$.

In the critical case $p=p_c$, $N\ge3$, there is a weak smoothing effect: solutions with an integrable initial datum will be for any positive time in~$L^1(\mathbb{R}^N)\cap M^{p_c}(\mathbb{R}^N)$, and we have the decay rate~$\|u(\cdot,t)\|_{M^{p_c}(\mathbb{R}^N)}\asymp t^{-\alpha}$; see once more~\cite{Kemppainen-Siljander-Vergara-Zacher-2016}. We complement this result here by proving that if $u_0\in L^1(\mathbb{R})\cap L^{p_c}(\mathbb{R}^N)$ then $\|u(\cdot,t)\|_{L^{p_c}(\mathbb{R}^N)}\asymp t^{-\alpha}$.

When $N=2$ the situation for the critical case, $p=\infty$, is somewhat different.  Indeed, assuming $u_0\in L^1(\mathbb{R}^2)\cap L^\infty(\mathbb{R}^2)$ we obtain in the present paper a rate  with a logarithmic correction, $\|u(\cdot,t)\|_{L^\infty(\mathbb{R}^2)}\asymp t^{-\alpha}\log t$.

Notice that,  in sharp contrast with the standard heat equation, when $\alpha\in(0,1)$ and $N\ge 3$ the decay rate increases with $p$ only up to the critical value $p=p_c$. Once this critical value is surpassed the decay rate remains unchanged, no matter the value of $p$, and is equal to the decay rate for the problem posed in a bounded domain with zero Dirichlet boundary condition, $\|u(\cdot,t)\|_{L^p(\Omega)}\asymp t^{-\alpha}$. This remarkable fact was already mentioned in~\cite{Kemppainen-Siljander-Vergara-Zacher-2016}, where it was  described as  a \emph{critical dimension phenomenon}: given $p\in(1,\infty)$ the decay rate increases with the dimension as long as $N<2p/(p-1)$, after which it stays equal to the one in bounded domains. For results on the decay rate for the problem in bounded domains see for instance~\cite{Dipierro-Valdinoci-Vespri-2019,Meerschaert-Nane-Vellaisamy-2009,Nakagawa-Sakamoto-Yamamoto-2010,Vergara-Zacher-2015}.

The critical phenomenon described above is better understood through the study of the decay rates in different space-time scales.  As we will see,  in the characteristic scale $|x|\asymp t^{\alpha/2}$, dictated by the scaling invariance of the equation, the decay rate is given by
$$
\|u(\cdot, t)\|_{L^p(\{\nu\le |x|/t^{\alpha/2}\le\mu \})}\asymp t^{-\frac{\alpha N}{2}(1-\frac1p)},\qquad 0<\nu<\mu<\infty,
$$
for all dimensions and all $p\in[1,\infty]$.
This could suggest a decay rate $\|u(\cdot,t)\|_{L^p(K)}\asymp t^{-\frac{\alpha N}2}$ in compact sets.  This is indeed the case when $N=1$. However, due to the memory effect introduced by the Caputo derivative, the decay rate in such sets is slower in higher dimensions,
$$
\|u(\cdot,t)\|_{L^p(K)}\asymp
\begin{cases}
t^{-\alpha},&N\ge3,\\
t^{-\alpha}\log t,&N=2.
\end{cases}
$$
In intermediate scales, $|x|\asymp g(t)$ with $g$ \emph{growing slowly to infinity},
\begin{equation}
\label{eq:intermediate.scales}
g(t)\to\infty,\quad g(t)=o(t^{\alpha/2}),
\end{equation}
we have an intermediate decay rate
$$
\|u(\cdot,t)\|_{L^p(\{\nu<|x|/g(t)<\mu\})}\asymp
\begin{cases}
t^{-\alpha}(g(t))^{2-N\big(1-\frac1p\big)},&N\ge3,\\
t^{-\alpha}(g(t))^{\frac2p}\log(g(t)t^{-\frac\alpha2}),&N=2.
\end{cases}
$$
In fast scales,
\begin{equation}
\label{eq:def.fast.scale}
|x|\asymp g(t)\quad\text{with }
g(t)t^{-\frac\alpha2}\to \infty,
\end{equation}
the decay rate is always faster than in the characteristic scale.
Thus,  the slowest decay rate takes place in the characteristic scale if $p$ is subcritical and in compact sets otherwise. If $p=p_c$, $N\ge 3$, solutions decay at the same rate in all scales starting from compact sets up to the characteristic scale.

\subsection{Asymptotic profiles}
 If~\eqref{eq:subcritical.range} holds, $Z(\cdot,t)\in L^p(\mathbb{R}^N)$ for all $t>0$, and one expects to have a convergence result analogous to~\eqref{eq:convergence.local}, namely,
\begin{equation}\label{eq:convergence.Lp}
t^{\frac{\alpha N}{2}(1-\frac1p)}\|u(\cdot,t)-M  Z(\cdot,t)\|_{L^p(\R^N)}\to0\quad\mbox{as }t\to\infty, \qquad M =\int_{\mathbb{R}^N}u_0.
\end{equation}
This was already proved in~\cite{Kemppainen-Siljander-Vergara-Zacher-2016} for  $p\in[1,N/(N-1))$ when $N\ge 2$ and $p\in[1,\infty)$ when $N=1$.  We fill up the gap, proving \eqref{eq:convergence.Lp} when $p=\infty$ and  $N=1$ and for all $p\in[1,p_c)$ if  $N\ge 2$ in Section~\ref{sect-global subcritical}.

Though the result is valid in the whole $\mathbb{R}^N$, in dimensions $N\ge2$ it is only sharp in the  characteristic scale $|x|\asymp t^{\alpha/2}$ (in dimension $N=1$ it is also sharp in inner scales, $|x|=o(t^{\alpha/2})$). The decay in other scales in this subcritical range is faster,  $o(t^{-\frac{\alpha N}2(1-\frac1p)})$. When properly rescaled to take into account the right rate of decay the solutions may approach a different asymptotic profile; see below.

The profile of the fundamental solution is smooth and belongs to all $L^p$  spaces in \emph{outer scales}, $|x|\asymp g(t)$ with $g(t)\ge\nu t^{\alpha/2}$, $\mu>0$.
Hence, the fundamental solution  may give the asymptotic behavior in such scales also when $p$ is not subcritical. We will prove that  indeed it does so, assuming certain (non-integrable) decay for the initial datum in addition to the integrability, namely,
$u_0\in \mathcal{D}_N$, where
$$
\mathcal{D}_\beta=\{f: \text{there exist } C,R>0 \text{ such that } |x|^\beta|f(x)|\le C\text{ for all } x\in \mathbb{R}^N, |x|>R\}.
$$
Thus, we have
\begin{equation}
	\label{eq:convergence.outer.regions}
	t^{\frac{\alpha N}2(1-\frac1p)}\|u(\cdot,t)-M  Z(\cdot,t)\|_{L^p(\{|x|>\nu t^{\alpha/2}\})}\to 0\quad\mbox{as }t\to\infty.
\end{equation}
As in the subcritical case, the result is only sharp in the characteristic scale. In faster scales it only says that the decay is $o(t^{-\frac{\alpha N}2(1-\frac1p)})$.

In intermediate scales, $|x|\asymp g(t)$, with $g$ satisfying~\eqref{eq:intermediate.scales},  the large-time behavior is still given by $M$ times the fundamental solution, under the additional assumption $u_0\in \mathcal{D}_N$ when $p$ is not subcritical. Indeed,
$$
\frac{\|u(\cdot,t)-M
Z(\cdot,t)\|_{L^p(\{\nu<|x|/g(t)<\mu\})}}{\|u(\cdot,t)\|_{L^p(\{\nu<|x|/g(t)<\mu\})}}\to0\quad\mbox{as }t\to\infty, \qquad0<\nu\le \mu<\infty.
$$
Note than in these scales $|\xi|=|x|t^{-\alpha/2}\to0$ as $t\to\infty$. Hence, for $t$ large
$$
Z(x,t)=t^{-\frac{\alpha N}2}F(\xi)\sim t^{-\frac{\alpha N}2}\kappa E_N(\xi)\sim\begin{cases}
t^{-\alpha}\kappa E_N(x), &N\ge3,\\
 t^{-\alpha}|\log(g(t)t^{-\alpha/2})|\kappa,&N=2,\\
t^{-\alpha/2}F(0),&N=1.
\end{cases}
$$
Therefore, in intermediate scales  we have, for $0<\nu\le \mu<\infty$,
\begin{equation}
\label{eq:behavior.intermediate.scales}
\begin{array}{ll}
\displaystyle(g(t))^{N\big(1-\frac 1p\big)-2}\|t^{\alpha} u(\cdot,t)-M\kappa E_N
\|_{L^p(\{\nu <|x|/g(t)<\mu\})}\to0, &N\ge3,\\[10pt]
\displaystyle g(t)^{-\frac 2p}\Big\|\frac{t^{\a}}{|\log(g(t)t^{-\alpha/2})|}u(\cdot,t)-M  \kappa\Big\|_{L^p(\{\nu <|x|/g(t)<\mu\})}\to0,\qquad&N=2,\\[10pt]
\displaystyle g(t)^{-\frac 1p}\|t^{\frac\a2}u(\cdot,t)-M  F(0)\|_{L^p(\{\nu <|x|/g(t)<\mu\})}\to0,&N=1,
\end{array}
\qquad\text{as }t\to\infty.
\end{equation}
Thus, the asymptotic profile is proportional to the fundamental solution for the Laplacian if $N\ge3$ and a constant for low dimensions. It is worth noticing that while the constant limit profile is the same one for all intermediate scales when $N=1$, it depends on the specific scale under consideration when $N=2$. For instance, if $g(t)=t^\gamma$, $\gamma\in (0,\alpha/2)$,
\begin{equation}
\label{eq:specific.intermediate.scale}
t^{-\frac {2\gamma}p}\Big\|\frac{t^{\a}}{\log t}u(\cdot,t)-M  \kappa\big(\frac\alpha2-\gamma\big)\Big\|_{L^p(\{\nu <|x|/g(t)<\mu\})}\to0\quad\mbox{as }t\to\infty.
\end{equation}
Another difference is that in the two-dimensional case the decay rate in $L^\infty$ is not the same for all intermediate scales. Thus, if $g(t)=t^{\alpha/2}/(\log t)^\theta$ we have
$$
\Big\|\frac{t^{\a}}{\log(\log t)}u(\cdot,t)-M  \kappa\theta\Big\|_{L^\infty(\{\nu <|x|/g(t)<\mu\})}\to0\quad\mbox{as }t\to\infty,
$$
which is to be compared with~\eqref{eq:specific.intermediate.scale}.

 In order to guess what may be the asymptotic behavior in compact sets, we start by considering the case of compactly supported initial data. In this situation, using the behavior of $F$ at the origin, we get for $N\ge3$ that locally
$$
u(x,t)=\int_{\mathbb{R}^N}u_0(x-y)t^{-\alpha N/2}F(yt^{-\alpha /2})\,dy\approx t^{-\alpha}\kappa \Phi(x),\qquad \Phi(x):=\int_{\mathbb{R}^N}\frac{u_0(x-y)}{|y|^{N-2}}\,dy.
$$
Thus, $t^\alpha u(\cdot,t)$ is expected to approach a multiple of the Newtonian potential of the initial datum in compact sets. Thus, the precise form of the initial datum plays a more important role in the asymptotic description than in the characteristic and intermediate scales, where it was only felt through the constant $M$. This remebrance of the initial datum is due to the  memory introduced by the nonlocal time derivative.

The result will not be restricted to compactly supported initial data. Indeed, if $u_0\in L^1(\mathbb{R}^N)$ if $p$ is subcritical, and $u_0\in  L^1(\mathbb{R}^N)\cap L_{\rm loc}^p(\mathbb{R}^N)$ otherwise, for   $N\ge3$ and any compact set $K$,
\begin{equation}
\label{eq:convergence.newtonian.potential}
\big\|t^\a u(\cdot,t)-\kappa\Phi\big\|_{L^p(K)}\to0\quad\mbox{as }t\to\infty.
\end{equation}
In the critical case $p=p_c$, in which solutions decay at the same rate in all scales starting from compact sets up to the characteristic scale, convergence can be extended to expanding balls $ B_{g(t)}$ with $g$ growing slowly to infinity,
\[
\|t^\a u(\cdot,t)-\kappa\Phi\|_{L^{p_c}(B_{g(t)})}\to0\quad\mbox{as }t\to\infty.
\]
An analogous result holds in the Marcinkiewicz space $M^{p_c}(B_{g(t)})$ assuming only $u_0\in   L^1(\R^N)$.

On the other hand, we will prove that
\begin{equation}
\label{eq:equivalence.Phi.EN.intermediate.scales}
(g(t))^{N\big(1-\frac 1p\big)-2}\|\Phi- ME_N\|_{L^p(\{\nu<|x|/g(t)<\mu\})}\to0\quad\mbox{as }t\to\infty, \qquad0<\nu\le \mu<\infty,
\end{equation}
for $N\ge3$, assuming $u_0\in L^1(\mathbb{R}^N)$ if $p$ is subcritical and $u_0\in L^1(\mathbb{R}^N)\cap\mathcal{D}_N$ otherwise, which combined with the behavior in intermediate escales~\eqref{eq:behavior.intermediate.scales} yields
$$
\displaystyle(g(t))^{N\big(1-\frac 1p\big)-2}\|t^{\alpha} u(\cdot,t)-\kappa \Phi
\|_{L^p(\{\nu <|x|/g(t)<\mu\})}\to0\quad\text{as }t\to\infty.
$$
Thus, the behavior in intermediate scales can be described both in terms of $MZ$ and $\kappa\Phi$, which shows the overlapping of both developments there.

As in intermediate scales, in low dimensions the asymptotic profile in compact sets  is a constant,
$$
\begin{array}{ll}
\displaystyle
\Big\|\frac{t^\a}{\log t}u(\cdot,t)-\frac{M \kappa\a}2\Big\|_{L^p(K)}\to0,\quad&N=2,\\[10pt]
\displaystyle\big\|t^{\frac\a2}u(\cdot,t)-M  F(0)\big\|_{L^p(K)}\to0, \quad&N=1,\\
\end{array}
\quad\mbox{as } t\to\infty,
$$
with the additional assumption $u_0\in L^1(\mathbb{R}^2)\cap L^q_{\rm loc}(\mathbb{R}^2)$ for some $q\in(1,\infty]$ when $N=2$ and $p=\infty$.
There is, however, a big difference between these two dimensions. While for $N=1$ the limit constant is the same for compact sets as for all intermediate scales, for $N=2$ the constant for compact sets only coincides with the one for intermediate scales if they  satisfy $\log g(t)/\log t\to 0$. These are precisely the scales for which the decay rate coincides with the one in compact sets. Thus, the overlapping for $N=1$ is much wider than for $N=2$. See Theorems~\ref{teo-compactos N=2 todo p} and~\ref{thm:compact.N=1} for more information on this subject.

\subsection{Fast scales. } We finally turn our attention to  \emph{fast scales}, satisfying~\eqref{eq:def.fast.scale}, for which the situation is more involved. As we will see, there is a difference between \emph{moderately fast} scales and \emph{very fast} scales. If the scales are moderately fast the asymptotic behavior of the solution is still given by $MZ$, with a decay rate that depends only on the scale. For very fast scales the behavior of the initial datum at infinity becomes more important, and may even determine the rate of decay and the final profile. Which scales are on each side, moderately fast or very fast, depends strongly on the decay of the initial datum at infinity.

We start by considering compactly supported initial data. We obtain convergence in uniform  relative error to $MZ$ in sets of the form $\{
\nu t^{\a/2}\le |x|\le\mu t(\log t)^{\frac{2-\alpha}\alpha}\}$
for all $\nu>0$ and for all $\mu\in (0,\mu^*)$, where $\mu^*$ depends on the support of $u_0$. Thus, in this case fast scales $g$ such that $g(t)<\mu^* t(\log t)^{\frac{2-\alpha}\alpha}$ are identified as moderately fast scales.

If the initial datum does not have compact support, but  has some integrable decay, namely  $u_0\in\mathcal{D}_\beta$ for some $\beta>N$,
we prove convergence in uniform relative error to $MZ$ in sets of the form $\{\nu t^{\a/2}\le |x|\le \mu t^{\a/2}(\log t)^{(2-\a)/2}\}$ for all $\nu>0$ and all $\mu\in (0,\mu_{\beta})$, where $\mu_{\beta}=\left(\frac\a{2\sigma}(\beta-N)\right)^{\frac{2-\alpha}2}$. The  upper restriction here is not technical. Indeed, if there exists $A>0$ such that $|x|^\beta u_0(x)\to A$ as $|x|\to\infty$, then
$$
u(x,t)=\frac A{|x|^\beta}(1+o(1))\quad\mbox{uniformly in }  |x|\ge  \nu t^{\a/2}(\log t)^{(2-\a)/2}.
$$
for every $\nu>\mu_\beta$. Notice that in this case moderately fast scales do not go so far as in the case  of compactly supported initial data.

The study of the behavior in fast scales is an interesting subject which is not complete even for the local heat equation.

\subsection{Final remarks}
We end this Introduction with some remarks and comments on the problem.

\noindent\emph{On the proofs. } Though the equation has a scaling invariance, it does not have a smoothing effect. Hence, the classical method to study large time behaviors that combines scaling with compactness cannot be used here.  Instead, we perform a careful analysis of the solution through the representation formula.
	
\noindent\emph{Solutions with sign changes. } Since the problem is linear,  the results in the present paper are valid for sign-changing solutions. However, if $\int_{\mathbb{R}^N}u_0=0$ they are not optimal: faster decay rates and different asymptotic profiles are expected in intermediate, characteristic, and moderately fast scales for all dimensions, and in compact sets in low dimensions, $N=1,2$. However, since the Newtonian potential of $u_0$ for $N\ge3$ is non-trivial if $u_0$ is not identically zero, the results in compact sets are optimal in high dimensions also in this case.

\noindent\emph{Global results. } Even though one may get global results if $p$ is not critical, they only give information either in the characteristic scale or in compact sets. That is the reason why in this paper we follow a different approach, considering asymptotic limits in different scales, instead of trying to obtain global profiles.

\noindent\emph{Comparison with the local case. } As we have already mentioned, we observe several important qualitative differences with the local case. The main one regarding the rates is that for the nonlocal case the $L^\infty$ decay rate changes when one passes from the characteristic scale to inner scales, becoming smaller, due to the memory. As for the asymptotic profiles, the difference is even more striking: in compact sets it is not given by a multiple of the fundamental solution anymore, but by the Newtonian potential of the initial datum, again as a byproduct of the memory effects introduced by the nonlocal time derivative.
	
\noindent\emph{Precedents. } The possibility of having $L^\infty$ decay rates in inner sets different from the one in the characteristic scale had already been observed in problems posed in domains with holes in low dimensions for different diffusion operators, including the local heat equation~\cite{Herraiz-1999}, the porous medium equation~\cite{Cortazar-Quiros-Wolanski-2018}, or certain nonlocal (in space) heat equations~\cite{Cortazar-Elgueta-Quiros-Wolanski-2016,Cortazar-Elgueta-Quiros-Wolanski-2016-b}. The decay is faster in low dimensions in inner scales due to the presence of the holes.  The nonlocal Cauchy problem~\eqref{eq:main} is the first example of this phenomenon for an equation  posed in the whole space that  we are aware of.  In this case the decay in inner sets is slower due to the memory, and takes place in high dimensions (and also in the critical dimension $N=2$). This is also, up to our knowledge, the first example of a diffusion problem in which the limit profile is given by the Newtonian potential.

\noindent\emph{Fully nonlocal problems. } In the forthcoming article~\cite{Cortazar-Quiros-Wolanski-2020-Preprint-1} the authors extend the results in the present paper to the fully nonlocal problem
$$
\partial_t^\alpha u+(-\Delta)^s u=0 \quad\text{in } \R^N\times\R_+,\qquad
u(\cdot,0)=u_0\quad\text{in } \R^N,
$$
with $s\in(0,1)$, an equation that has received recently much attention; see for instance~\cite{Allen-Caffarelli-Vasseur-2016}. The analysis is based in the estimates for the fundamental solution of the problem provided by the significant paper~\cite{Kemppainen-Siljander-Zacher-2017}, where some partial results on the large time behavior were also obtained.

\noindent\emph{The non-homogeneus problem. } Another interesting issue is the large-time behavior of solutions to the non-homogeneous problem
$$
\partial_t^\alpha u-\Delta u=f \quad\text{in } \R^N\times\R_+,\qquad
u(\cdot,0)=0\quad\text{in } \R^N,
$$
for certain choices of the right-hand-side function $f$. Some partial results are available in~\cite{Kemppainen-Siljander-Zacher-2017}. A more complete study of this subject is performed in the upcoming paper~\cite{Cortazar-Quiros-Wolanski-2020-Preprint-2}.

\noindent\emph{On the decay condition. } It can be checked without much effort that the  condition $u_0\in\mathcal{D}_N$ that appears in several theorems when $p$ is not subcritical can be substituted by the weaker integral condition $u_0\in\mathcal{D}_{N,p}$, where
$$
\mathcal{D}_{N,p}=\{f: \text{there exist } C,R>0 \text{ such that } r^{N\big(1-\frac1p\big)}\|f\|_{L^p(\{|x|>r\})}\le C\text{ for all }r\ge R\}.
$$
Note that when $p=\infty$ both classes coincide, $\mathcal{D}_N=\mathcal{D}_{N,\infty}$.

\noindent\emph{Organization of the paper. } We devote Section~\ref{sect-global subcritical} to prove that $MZ$ gives the global asymptotic behavior for all subcritical values of $p$. The next sections are devoted to obtain rates and profiles in different scales, classified depending on their speed. Outer scales (which include the characteristic scale) are considered in Section~\ref{sect-diffusive scale}, Intermediate scales in Section~\ref{sect-slow scales}, and compact sets in Section~\ref{sect-compact sets}. Fast scales, for which only partial results are available, even in the local case, are considered in Section~\ref{sect-fast scales}.

\section{Global results in the subcritical range}
\label{sect-global subcritical} \setcounter{equation}{0}

In this section we consider the large time behavior in $L^p(\mathbb{R}^N)$ when $p$ is in the subcritical range.
We will prove that the behavior is given by $M Z$, thus extending to the whole subcritical range the results from~\cite{Kemppainen-Siljander-Vergara-Zacher-2016} for $p\in[1,N/(N-1))$ when $N\ge 2$ and for $p\in[1,\infty)$ when $N=1$. The result is sharp in the characteristic scale for $N\ge 2$ and in all inner scales for $N=1$.
\begin{teo}\label{L^p} Let $u_0\in L^1(\R^N)$. If $p$ is subcritical, then
	\begin{equation}\label{lp}
	t^{\frac{\alpha N}{2}(1-\frac1p)}\big\|u(\cdot,t)-M Z(\cdot,t)\big\|_{L^p(\R^N)}\to0\quad\mbox{as }t\to\infty.
	\end{equation}
	
\end{teo}

\begin{proof}
The starting point is the identity
\begin{equation}
\label{eq:starting.point}
t^{\frac{\alpha N}2}\big(u(x,t)-M  Z(x,t)\big)
=\int u_0(y)\big(F((x-y)t^{-\a/2})-F(xt^{-\a/2})\big)\,dy.
\end{equation}
Hence, if $p$ is subcritical and finite, performing the change of variables $\xi=xt^{-\alpha/2}$,
\[
\begin{aligned}
t^{\frac{\alpha N}{2}(1-\frac1p)}\|u(\cdot,t)&-M  Z(\cdot,t)\|_{L^p(\R^N)}\le \int u_0(y)\left(\int\big|F(\xi-yt^{-\a/2})-F(\xi)\big|^p\,d\xi\right)^{1/p}\,dy
\\
&={\rm I}+{\rm II},\text{where }\\	
{\rm I}&=\int_{|y|<\delta t^{\a/2}} u_0(y)\left(\int\big|F(\xi-yt^{-\a/2})-F(\xi)\big|^p\,d\xi\right)^{1/p}\,dy,\\
	{\rm II}&=\int_{|y|>\delta t^{\a/2}} u_0(y)\left(\int\big|F(\xi-yt^{-\a/2})-F(\xi)\big|^p\,d\xi\right)^{1/p}\,dy.
	\end{aligned}
	\]
	Now, since $p$ is finite, given $\ep>0$  we can take $\delta>0$ small enough so that
	\[
	\Big(\int\big|F(\xi-yt^{-\a/2})-F(\xi)\big|^p\,d\xi\Big)^{1/p}<\ep\quad\mbox{if }|y|t^{-\a/2}<\delta.
	\]
	With this choice of $\delta$ we have  ${\rm I}\le M\ep$.
	
	On the other hand, if $t_0$ is large enough,
	\[
	{\rm II}\le 2\|F\|_{L^p(\R^N)}\int_{|y|>\delta t^{\a/2}}u_0(y)\,dy\le2\|F\|_{L^p(\R^N)}\,\ep\quad\mbox{if } t\ge t_0,
	\]
	and the theorem is proved for $p$ subcritical and finite.
	
 If $N=1$  and $p=\infty$, then  $t^{\frac{\alpha }2}\big|u(x,t)-M  Z(x,t)\big|\le {\rm I}+{\rm II}$, where
	\[
	\begin{aligned}
	{\rm I}&=
	\int_{|y|<\delta(t)t^{\a/2}}u_0(y)|F((x-y)t^{-\a/2})-F(xt^{-\a/2})|\,dy,
	\\
	{\rm II}&=\int_{|y|>\delta(t)t^{\a/2}}u_0(y)|F((x-y)t^{-\a/2})-F(xt^{-\a/2})|\,dy.
	\end{aligned}
	\]
	Since in this case $F$ is uniformly Lipschitz continuous and bounded, then
	\[
	\begin{aligned}
	{\rm I}&\le \|DF\|_{L^\infty(\mathbb{R})}\int_{|y|<\delta(t)t^{\a/2}}u_0(y)|y|t^{-\a/2}\,dy\le  \|DF\|_{L^\infty(\mathbb{R})}
	\|u_0\|_{L^1(\mathbb{R})}\delta(t),\\
	{\rm II}&\le 2\|F\|_{L^\infty(\mathbb{R})}\int_{|y|>\delta(t)t^{\a/2}}u_0(y)\,dy,
	\end{aligned}
	\]
	and the result follows by choosing $\delta(t)\to0$ with $\delta(t)t^{\a/2}\to\infty$ as $t\to\infty$.
\end{proof}

\section{Outer scales}
\label{sect-diffusive scale} \setcounter{equation}{0}

We prove now that the fundamental solution  gives the asymptotic behavior in outer sets also when $p$ is not subcritical, assuming some decay for the initial datum. This gives the sharp rate of decay in the characteristic scale.
\begin{teo}\label{teo-supercritical exterior} Let $N\ge 2$ and $p\in [p_c,\infty]$. Assume $u_0\in L^1(\R^N)\cap \mathcal{D}_N$. Then, for any $\nu>0$,
\[
t^{\frac{\alpha N}2(1-\frac1p)}\|u(\cdot,t)-M  Z(\cdot,t)\|_{L^p(\{|x|>\nu t^{\alpha/2}\})}\to 0\quad\mbox{as }t\to\infty.
\]
\end{teo}
\begin{proof}
Starting from~\eqref{eq:starting.point} we get, for $N\ge 3$, $p\in [p_c,\infty)$,
\[\begin{aligned}
t^{\frac{\alpha N}2(1-\frac1p)}\big\|u(\cdot,t)&-M  Z(\cdot,t)\big\|_{L^p(\{|x|>\nu t^{\alpha/2}\})}
\le {\rm I}+{\rm II}+{\rm III},\quad\text{where }\\
{\rm I}&=t^{-\frac{\alpha N}{2p}}\int u_0(y)\Big(\int_{\stackrel{|x|> \nu t^{\alpha/2}}{|y|<\delta|x|}}\big|F((x-y)t^{-\a/2})-F(xt^{-\a/2})\big|^p\,dx\Big)^{1/p}\,dy,\\
{\rm II}&=t^{-\frac{\alpha N}{2p}}
\left(\int_{|x|> \nu t^{\alpha/2}}\Big(\int_{|y|>\delta|x|} u_0(y)F(xt^{-\a/2})\,dy\Big)^p\,dx\right)^{1/p},\\
{\rm III}&=t^{-\frac{\alpha N}{2p}} \left(\int_{|x|> \nu t^{\alpha/2}}\Big(\int_{|y|>\delta|x|} u_0(y)F((x-y)t^{-\a/2})\,dy\Big)^p\,dx\right)^{1/p}.
\end{aligned}
\]
Performing the change of variables $\xi=x t^{-\alpha/2}$ we get
$$
{\rm I}=
\int u_0(y)\Big(\int_{\stackrel{|\xi|>\nu}{|y|t^{-\alpha/2}<\delta|\xi|}}\big|F(\xi-yt^{-\alpha/2})-F(\xi)\big|^p\,d\xi\Big)^{1/p}\,dy.
$$
Let $\delta\in(0,1/2)$. If  $|\xi|\ge \nu $ and $|y|t^{-\alpha/2}\le\delta|\xi|$, then  $|\xi-syt^{-\a/2}|\ge |\xi|/2\ge\nu/2$ for all $s\in[0,1]$. Therefore, using~\eqref{eq:estimate.MVT},
\[
{\rm I}\le C M\delta\Big(\int_{|\xi|>\nu}\big(\textrm{e}^{-\bar\sigma|\xi|^{2/(2-\a)}}|\xi|)^p\,d\xi\Big)^{1/p}\le C \delta,\quad \text{where }\bar\sigma=\sigma 2^{-2/(2-\alpha)}.
\]

On the other hand,
\[
{\rm II}\le  \left(\int_{|\xi|>\nu}|F(\xi)|^p\,d\xi\right)^{1/p}\int_{|y|>\delta \nu t^{\a/2}}u_0(y)\,dy =C\int_{|y|>\delta\nu t^{\a/2}}u_0(y)\,dy.
\]

As for ${\rm III}$, we take $\ep\in(0,1)$ and $\gamma> 0$ to be chosen later and make also the change of variables $\eta=yt^{-\alpha/2}$. Since $u_0\in\mathcal{D}_N$, using the estimates~\eqref{eq:global.bounds.F}--\eqref{eq:bounds.F.origin} for the profile $F$ we get
\[\begin{aligned}
{\rm III}
=
&\left(\int_{|\xi|> \nu}\Big(t^{\frac{\alpha N}{2}}\int_{|\eta|>\delta|\xi|} u_0(\eta t^{\a/2})F(\xi-\eta)\,d\eta\Big)^p\,d\xi\right)^{1/p}
\le {\rm III}_{1}+{\rm III}_{2},\quad\text{where }\\
&{\rm III}_{1}= C\left(\int_{|\xi|\ge \nu}\Big(\int_{\stackrel{|\eta|>\delta|\xi|}{|\xi-\eta|<\gamma|\xi|^\ep}}|\xi-\eta|^{2-N}|\eta|^{-N}\,d\eta\Big)^p\,d\xi\right)^{1/p},\\
&{\rm III}_{2}= C\left(\int_{|\xi|\ge \nu}\Big( t^{\frac{\a N}2}\int_{\stackrel{|\eta|>\delta|\xi|}{|\xi-\eta|>\gamma|\xi|^\ep}}u_0(t^{\a/2}\eta)e^{-\sigma|\xi-\eta|^{\frac2{2-\a}}}\,d\eta\Big)^p\,d\xi\right)^{1/p}.
\end{aligned}\]
Since $p\ge \frac N{N-2}$, if we choose  $\gamma=\delta^{\frac{N+1}2}$ we have
\[\begin{aligned}
{\rm III}_{1}&\le C \delta^{-N}\left(\int_{|\xi|\ge \nu}|
\xi|^{-Np}\Big(\int_{|\xi-\eta|<\gamma|\xi|^\ep}|\xi-\eta|^{2-N}\,d\eta\Big)^p\,d\xi\right)^{1/p}\\
&\le C \delta^{-N}\gamma^2\left(\int_{|\xi|\ge \nu}|\xi|^{-p(N-2\ep)}\,d\xi\right)^{1/p}\le C \delta.
\end{aligned}\]

On the other hand,
\[\begin{aligned}
{\rm III}_{2}&\le C\left(\int_{|\xi|\ge \nu}\textrm{e}^{-p\sigma\gamma^{\frac2{2-\a}}
|\xi|^{\frac{2\ep}{2-\a}}}\Big(t^{\frac{N\a}2}\int_{|\eta|>\delta|\xi|}u_0(\eta
t^{\a/2})\,d\eta\Big)^p\,d\xi\right)^{1/p}\\
&\le C\left(\int_{|\xi|\ge \nu}\textrm{e}^{-p\sigma\gamma^{\frac2{2-\a}}
	|\xi|^{\frac{2\ep}{2-\a}}}\,d\xi\right)^{1/p}\int_{|y|>\delta\nu t^{\a/2}}u_0(y)\,dy\to 0\quad \text{as }t\to\infty,
\end{aligned}
\]
which completes the proof for $N\ge3$ and $p$ finite.

The cases $N\ge 3$, $p=\infty$ follow by letting $p\to\infty$ in the estimates for finite $p$.

In order to deal with the case $N=2$, $p=\infty$, starting from~\eqref{eq:starting.point} we get
\[
\begin{aligned}
 t^{\alpha}\big|u(x,t)&-M  Z(x,t)\big|\le {\rm I}+{\rm II}+{\rm III},\quad\text{where }\\
{\rm I}&=
\int_{|y|<\delta t^{\a/2}}u_0(y)|F((x-y)t^{-\a/2})-F(xt^{-\a/2})|\,dy,
\\
{\rm II}&=\int_{|y|>\delta t^{\a/2}}u_0(y)F(xt^{-\a/2})\,dy,
\\
{\rm III}&=\int_{|y|>\delta t^{\a/2}}u_0(y)F((x-y)t^{-\a/2})\,dy.
\end{aligned}
\]

If  $|x|\ge \nu t^{\alpha/2}$ and $|y|\le\delta t^{\alpha/2}$, with $\delta<\nu/2$ to be chosen, then $|xt^{-\alpha/2}-syt^{-\alpha/2}|\ge \nu/2$ for all  $s\in[0,1]$.
Therefore, since $F$ is uniformly Lipschitz continuous in $\mathbb{R}^2\setminus B_{\nu/2}$,
\[
|F((x-y)t^{-\a/2})-F(xt^{-\a/2})|\le|DF(xt^{-\alpha/2}-syt^{-\alpha/2})| \,|y|t^{-\alpha/2}
\le \| DF\|_{L^\infty(\mathbb{R}^2 \setminus B_{\nu/2})}\delta <\ep
\]
if $\delta<\min\{\nu/2, \ep/\| DF\|_{L^\infty(\mathbb{R}^2 \setminus B_{\nu/2})}\}$, so that $\textrm{I}\le M\ep$.

We now bound $\textrm{II}$. As $|x|t^{-\alpha/2}\ge \nu$ and $F$ is bounded outside $B_\nu$,
\[
\textrm{II}\le \displaystyle
\displaystyle \|F\|_{L^\infty(\mathbb{R}^2\setminus B_\nu)}
\int_{|y|>\delta t^{\a/2}}u_0(y)\,dy <\ep
\quad\text{for }t\ge t_0(\ep,\delta,\nu).
\]

Now we turn to $\textrm{III}$. We have $\textrm{III}=\textrm{III}_{1}+\textrm{III}_{2}$, where
\[
\begin{aligned}
 \textrm{III}_{1}&=\int_{\stackrel{|y|>\delta t^{\alpha/2}}{|x-y|<\gamma t^{\alpha/2}}}|y|^2u_0(y)\frac{F((x-y)t^{-\alpha/2})}{|y|^2}\,dy,\\
\textrm{III}_{2}&=\int_{\stackrel{|y|>\delta t^{\alpha/2}}{|x-y|>\gamma t^{\alpha/2}}}u_0(y)F((x-y)t^{-\alpha/2})\,dy.
\end{aligned}
\]
Since $u_0\in\mathcal{D}_2$, using  the bound $F(\xi)\le C|\log|\xi||\le C|\xi|^{-1/2}$ for $|\xi|\le 1/2$, we get
\[
\textrm{III}_{1}\le  C\delta^{-2}t^{-\alpha}\int_{|x-y|<\gamma t^{\alpha/2}}{|\log|(x-y)t^{-\alpha/2}||}\,dy\le C \delta^{-2}\gamma^{3/2}
\]
if $\gamma\le 1/2$. Therefore, if we choose $\gamma=\delta^{2}$ with $\delta\in(0,1/\sqrt2)$ we finally get that
$\textrm{III}_{1}\le C\delta.$

On the other  hand,
\[
\textrm{III}_{2}\le \|F\|_{L^\infty(\mathbb{R}^2\setminus B_{\delta^2})}\int_{|y|>\delta t^{\a/2}}u_0(y)\,dy<\ep\quad\mbox{if } t\ge t_0(\ep,\delta).
\]
\end{proof}

\section{Intermediate scales}
\label{sect-slow scales} \setcounter{equation}{0}

In this section we consider intermediate scales, $|x|\asymp g(t)$ with $g$ satisfying~\eqref{eq:intermediate.scales}, obtaining both sharp rates and limit profiles. As in the characteristic scale, solutions approach $MZ$, though with a different rate.
We consider high and low dimensions separately.

\subsection{High dimensions} We start with the case $N\ge 3$.
\begin{teo}\label{teo-intermedio sin orden}
Let $N\ge 3$ and  $\kappa>0$ as in~\eqref{eq:constant.Nge2}.  Let $u_0\in L^1(\mathbb{R}^N)$ if $p$ is subcritical, or $u_0\in L^1(\mathbb{R}^N)\cap \mathcal{D}_N$ if $p\in[p_c,\infty]$. For any   $g$ satisfying~\eqref{eq:intermediate.scales} and $0<\nu<\mu<\infty$,
\begin{equation}\label{eq-intermedio-p-1}
(g(t))^{N\big(1-\frac 1p\big)-2}\|t^{\alpha} u(\cdot,t)-M\kappa E_N
\|_{L^p(\{\nu <|x|/g(t)<\mu\})}\to0\quad\mbox{as } t\to\infty.
\end{equation}
As a consequence,
\begin{equation}\label{eq-intermedio-p-2}
t^{\alpha} (g(t))^{N\big(1-\frac 1p\big)-2}\|u(\cdot,t)-M
Z(\cdot,t)\|_{L^p(\{\nu<|x|/g(t)<\mu\})}\to0\quad\mbox{as }t\to\infty.
\end{equation}
\end{teo}
\begin{proof}
Since $u(\cdot,t)=Z(\cdot,t)*u_0$,   we have
\[
t^{\alpha}u(x,t)=t^{\alpha}\int Z(y,t)u_0(x-y)\,dy=t^{\frac\alpha2(2-N)}\int F(yt^{-\a/2})u_0(x-y)\,dy.
\]
Adding and subtracting $\displaystyle \kappa\int_{|x-y|<\delta|x|}\frac{u_0(x-y)}{|y|^{N-2}}\,dy$, we obtain the estimate
$$
\begin{aligned}
|t^{\alpha}u(x,t)-M\kappa E_N(x)|&\le{\rm I}(x,t)+{\rm II}(x,t)+{\rm III}(x,t)+{\rm IV}(x,t),\quad\text{where}\\
{\rm I}(x,t)&=
\int_{|x-y|<\delta|x|}\Big|\big(|y|t^{-\alpha/2}\big)^{N-2}F(yt^{-\alpha/2})-\kappa\Big|\frac{u_0(x-y)}{|y|^{N-2}}\,dy,\\
{\rm II}(x,t)&=
\kappa\int_{|x-y|<\delta|x|}\big||y|^{2-N}-|x|^{2-N}\big|u_0(x-y)\,dy,\\
{\rm III}(x,t)&=|x|^{2-N}\kappa \int_{|x-y
	|>\delta|x|}u_0(x-y)\,dy,\\
{\rm IV}(x,t)&=t^{\frac\alpha2(2-N)}\int_{|x-y|>\delta|x|} F(yt^{-\a/2})u_0(x-y)\,dy.
\end{aligned}
$$

Using the estimate for the profile~\eqref{eq:estimate.error.behaviour.profile.N=3} we get
\[
{\rm I}(x,t)\le  C
 t^{-\alpha/2}\int_{|x-y|<\delta|x|}\frac{u_0(x-y)}{|y|^{N-3}}\,dy.
\]
Since $|x-y|<\delta|x|$ implies that  $|y|\ge
|x|/2$    for any $\delta\in (0,1/2)$, then
${\rm I}(x,t)\le  C M t^{-\alpha/2}|x|^{3-N}$.
Thus, since $g(t)=o(t^{-\alpha/2})$,
$$
(g(t))^{N\big(1-\frac 1p\big)-2}\|{\rm I}(\cdot,t)\|_{L^p(\{\nu <|x|/g(t)<\mu\})}\le C M \mu^{N/p}\nu^{3-N}g(t)t^{-\alpha/2}\to0\quad\text{as }t\to\infty.
$$

On the other hand, if $|y|\ge |x|/2$, the Mean Value Theorem yields, for some $s\in[0,1]$,
$$
\big||y|^{2-N}-|x|^{2-N}\big|=
(N-2)(s|x|+(1-s)|y|)^{1-N}\big||x|-|y|\big|\le C|x|^{1-N}|x-y|.
$$
Therefore, ${\rm II}(x,t)\le C\kappa M |x|^{2-N}\delta$  for any $\delta\in (0,1/2)$, and hence
$$
(g(t))^{N\big(1-\frac 1p\big)-2}\|{\rm II}(\cdot,t)\|_{L^p(\{\nu <|x|/g(t)<\mu\})}\le C \kappa M\mu^{N/p}\nu^{2-N}\delta<\varepsilon \quad \text{for all } t\text{ if } \delta \text{ is small enough}.
$$

For any $\delta$ fixed we have ${\rm III}(x,t)\le (\nu g(t))^{2-N}\kappa\int_{|z|\ge\delta\nu g(t)}u_0(z)\,dz$ if $|x|\ge \nu g(t)$. Thus,
since $u_0\in L^1(\mathbb{R}^N)$ and $g(t)\to\infty$ as $t\to\infty$, we get ${\rm III}(x,t)\le (\nu g(t))^{2-N}\kappa\varepsilon$ for all $|x|\ge \nu g(t)$ and $t$ large enough, how large depending on $\delta$ and $\varepsilon$. Hence,
$$
(g(t))^{N\big(1-\frac 1p\big)-2}\|{\rm III}(\cdot,t)\|_{L^p(\{\nu <|x|/g(t)<\mu\})}\le C \kappa  \mu^{N/p}\nu^{2-N}\varepsilon\quad \text{for all } t\text{ large }.
$$

In order to estimate {\rm IV}, we use the bound $F(\xi)\le C|\xi|^{2-N}$ and split the domain of integration to obtain ${\rm IV}\le{\rm IV}_1+{\rm IV}_2$, where
\begin{equation}\label{eq-II}
{\rm IV}_1(x,t)=C\int_{\scriptsize\begin{array}{c}|y|<\gamma|x|\\|x-y|>\delta|x|\end{array}} \frac{u_0(x-y)}{|y|^{N-2}}\,dy,\qquad
{\rm IV}_2(x,t)=C\int_{\scriptsize\begin{array}{c}|y|>\gamma|x|\\|x-y|>\delta|x|\end{array}} \frac{u_0(x-y)}{|y|^{N-2}}\,dy.
\end{equation}
Let $p$ be subcritical. By H\"older's inequality,
$$
{\rm IV}_1(x,t)\le C \left(\int_{\scriptsize\begin{array}{c}|y|<\gamma|x|\\|x-y|>\delta|x|\end{array}} u_0(x-y)\,dy\right)^{\frac{p-1}p}\left(\int_{\scriptsize\begin{array}{c}|y|<\gamma|x|\\|x-y|>\delta|x|\end{array}} \frac{u_0(x-y)}{|y|^{(N-2)p}}\,dy\right)^{\frac1p}.
$$
Therefore, since $p$ is subcritical,
$$
\begin{aligned}
(g(t))^{N\big(1-\frac 1p\big)-2}&\|{\rm IV}_1(\cdot,t)\|_{L^p(\{\nu <|x|/g(t)<\mu\})}\\
&
\le (g(t))^{N\big(1-\frac 1p\big)-2}\left(\int_{|y|<\gamma\mu g(t)} |y|^{(2-N)p}\,dy\right)^{\frac1p}\int_{|z|>\delta\nu g(t)} u_0(z)\,dz\\
&\le C(\gamma\mu)^{2-N\big(1-\frac 1p\big)}\int_{|z|>\delta\nu g(t)} u_0(z)\,dz,
\end{aligned}
$$
which is small for any $\gamma,\delta>0$ fixed, if $t$ is large enough.

If $p$ is not subcritical,  we use the decay property $u_0\in \mathcal{D}_N$. Thus, if we choose $\gamma=\delta^{\frac{N+1}2}$,
\[
\begin{aligned}
{\rm IV}_1(x,t)&\le C \int_{\scriptsize\begin{array}{c}|y|<\gamma|x|\\|x-y|>\delta|x|\end{array}}\frac{dy}{|y|^{N-2}|x-y|^N}\le C|x|^{-N}\delta^{-N}\int_{|y|<\gamma |x|} |y|^{2-N}\,dy=C|x|^{2-N}\delta,
\end{aligned}
\]
and hence
$$
(g(t))^{N\big(1-\frac 1p\big)-2}\|{\rm IV}_1(\cdot,t)\|_{L^p(\{\nu <|x|/g(t)<\mu\})}\le  C\mu^{N/p}\nu^{2-N}\delta<\varepsilon,
$$
if $\delta$ is small enough.

Finally, for $\delta$ and $\gamma$ fixed, $
{\rm IV}_2(x,t)\le C  (\gamma|x|)^{2-N}\int_{|z|>\delta c g(t)}u_0(z)\,dz$.
Therefore,   for all $p\in[1,\infty]$,
$$
(g(t))^{N\big(1-\frac 1p\big)-2}\|{\rm IV}_2(\cdot,t)\|_{L^p(\{\nu <|x|/g(t)<\mu\})}\le  C\mu^{N/p}(\gamma\nu)^{2-N}\int_{|z|>\delta c g(t)}u_0(z)\,dz<\varepsilon,
$$
for $t$ large enough, since $g(t)\to\infty$ as $t\to\infty$, and $u_0\in L^1(\mathbb{R}^N)$, which ends the proof of~\eqref{eq-intermedio-p-1}.

We now proceed to prove~\eqref{eq-intermedio-p-2}.  The result will follow from the estimate~\eqref{eq:estimate.error.behaviour.profile.N=3}. Indeed,
\[
|t^\alpha Z(x,t)- \kappa E_N(x)|=|x|^{2-N}\big|(|x|t^{-\alpha/2})^{N-2}F(xt^{-\alpha/2})-\kappa \big|\le C |x|^{3-N}t^{-\alpha/2}.
\]
Therefore
$$
(g(t))^{N\big(1-\frac 1p\big)-2}\|t^\alpha Z(\cdot,t)- \kappa E_N\|_{L^p(\{\nu<|x|/g(t)<\mu\})}\le C\mu^{N/p}\nu^{3-N}g(t)t^{-\alpha/2},
$$
which combined with~\eqref{eq-intermedio-p-1} yields the result, since $g(t)=o(t^{\alpha/2})$.
\end{proof}

\subsection{Low dimensions} We now turn our attention to  the two dimensional case, where we find logarithmic corrections in the decay rates. In contrast with the case of higher dimensions, the asymptotic profile is a constant. However, which one depends on the precise intermediate scale under consideration.
\begin{teo}\label{teo-intermedio sin orden N=2}
Let $N=2$ and  $\kappa>0$ as in~\eqref{eq:constant.Nge2}.
Let $u_0\in L^1(\mathbb{R}^2)$ if $p$ is subcritical, or $u_0\in L^1(\mathbb{R}^2)\cap \mathcal{D}_2$ if $p=\infty$. For any   $g$ satisfying~\eqref{eq:intermediate.scales} and $0<\nu<\mu<\infty$ there holds
\begin{equation}\label{eq-intermedio1 N=2}
g(t)^{-\frac 2p}\Big\|\frac{t^{\a}}{|\log(g(t)t^{-\alpha/2})|}u(\cdot,t)-M  \kappa\Big\|_{L^p(\{\nu <|x|/g(t)<\mu\})}\to0\quad\mbox{as } t\to\infty.
\end{equation}
As a consequence,
\begin{equation}\label{eq-intermedio2 N=2}
\frac{g(t)^{-
		\frac2p}t^{\a}}{|\log (g(t)t^{-\alpha/2})|}
\|{u(\cdot,t)-M  Z(\cdot,t)}\|_{L^p(\{\nu <|x|/g(t)<\mu\})}\to0\quad\mbox{as }t\to\infty.
\end{equation}
\end{teo}
\begin{proof}
Proceeding similarly to what we did in the proof of Theorem~\ref{teo-intermedio sin orden} we get the estimate
$$
\begin{aligned}
\Big|&\frac{t^\alpha}{|\log(g(t)t^{-\alpha/2})|}u(x,t)-M  \kappa\Big|
\le{\rm I}(x,t)+{\rm II}(x,t)+{\rm III}(x,t)+{\rm IV}(x,t),\quad\text{where}\\	
{\rm I}(x,t)&=\frac{1}{|\log( g(t)t^{-\alpha/2})|}\int_{|x-y|<\delta|x|}\left|\frac{F(yt^{-\alpha/2})}{\big|\log(|y|t^{-\alpha/2})\big|}-\kappa\right|\big|\log(|y|t^{-\alpha/2})\big|
u_0(x-y)\,dy,\\
{\rm II}(x,t)&=\kappa\int_{|x-y|<\delta|x|}\left|\frac{\big|\log(|y| t^{-\alpha/2})\big|}{|\log (g(t)t^{-\alpha/2})|}-1\right|u_0(x-y)\,dy,\\
{\rm III}(x,t)&=\kappa\int_{|x-y|>\delta|x|}u_0(x-y)\,dy,\\
{\rm IV}(x,t)&=\frac1{|\log(g(t)t^{-\alpha/2})|}\int_{|x-y|>\delta|x|} F(yt^{-\a/2})u_0(x-y)\,dy.
\end{aligned}
$$

Let  $\nu <|x|/g(t)<\mu$ and $|x-y|<\delta|x|$. If $\delta\in (0,1/2)$, then  $|x|/2\le |y|\le 3|x|/2$, and hence
\begin{equation*}
\label{eq:bounds.region.under.consideration}
\frac\nu2 g(t)t^{-\alpha/2}< |y|t^{-\alpha/2}<\frac{3\mu}2
g(t)t^{-\alpha/2}.
\end{equation*}
In particular, $|y|t^{-\alpha/2}\to 0$ if $g(t)=o(t^{-\alpha/2})$. Moreover,
$$
1-\frac{|\log(3\mu/2)|}{|\log(g(t)t^{-\alpha/2}|}\le\frac{\big|\log(|y|t^{-\alpha/2})\big|}{|\log( g(t)t^{-\alpha/2})|}\le1+\frac{|\log(\nu/2)|}{|\log(g(t)t^{-\alpha/2})|},
$$
and therefore
\begin{equation}
\label{eq:logs.approach.1}
\frac{\big|\log(|y|t^{-\alpha/2})\big|}{|\log( g(t)t^{-\alpha/2})|}=1+o(1)\quad\text{as }t\to\infty.
\end{equation}
Using also the behavior of $F$ at the origin~\eqref{eq:constant.Nge2}, we get, for $t$ large enough, ${\rm I}(x,t)\le C\ep$.

The estimate \eqref{eq:logs.approach.1} also yields ${\rm II}(x,t)\le M \kappa\ep$.

On the other hand, since $g(t)\to\infty$ and $u_0\in L^1(\mathbb{R}^2)$,
\[
{\rm III}(x,t)\le \kappa\int_{|z|>\delta\nu g(t)}u_0(z)\,dz\to 0\quad\text{as }t\to\infty.
\]
We therefore have
$$
(g(t))^{-\frac 2p}\|({\rm I +II
	+III})(\cdot,t)\|_{L^p(\{\nu <|x|/g(t)<\mu\})})\le C\varepsilon\quad\text{for }t\text{ large}.
$$

In order to estimate ${\rm IV}$, we make the decomposition ${\rm IV}(x,t)={\rm IV}_1(x,t)+{\rm IV}_2(x,t)$, where
\[\begin{aligned}
&{\rm IV}_1(x,t)=\frac1{|\log (g(t)t^{-\alpha/2})|}\int_{\scriptsize\begin{array}{c}|y|<\gamma|x|\\|x-y|>\delta|x|\end{array}} F(yt^{-\a/2})u_0(x-y)\,dy,\\
&{\rm IV}_2(x,t)=\frac1{|\log (g(t)t^{-\alpha/2})|}\int_{\scriptsize\begin{array}{c}|y|>\gamma|x|\\|x-y|>\delta|x|\end{array}} F(yt^{-\a/2})u_0(x-y)\,dy.
\end{aligned}
\]
Let $p\in[1,\infty)$. Then,  using the behavior~\eqref{eq:constant.Nge2} and passing to radial coordinates,
$$
\begin{aligned}
(g(t))^{-\frac 2p}&\|{\rm IV}_1(\cdot,t)\|_{L^p(\{\nu <|x|/g(t)<\mu\})}\\
&
\le\frac1{(g(t))^{\frac 2p}|\log (g(t)t^{-\alpha/2})|} \left(\int_{|y|<\gamma\mu g(t)} F^p(yt^{-\a/2})\,dy\right)^{\frac1p}\int_{|z|>\delta\nu g(t)} u_0(z)\,dz\\
&\le\frac C{(g(t))^{\frac 2p}|\log (g(t)t^{-\alpha/2})|} \left(t^\alpha\int_0^{\gamma\mu g(t)t^{-\alpha/2}} r|\log r|^{p}\,dr\right)^{\frac1p}\int_{|z|>\delta\nu g(t)} u_0(z)\,dz\\
\end{aligned}
$$
Therefore, since $\int_0^sr|\log r|^p\,ds\le C s^2|\log s|^p$ for all $s\in (0,\frac12)$, we conclude that
$$
(g(t))^{-\frac 2p}\|{\rm IV}_1(\cdot,t)\|_{L^p(\{\nu <|x|/g(t)<\mu\})}\le \frac{C\gamma^2\mu^2 |\log(\gamma\mu g(t)t^{-\alpha/2})|}{|\log (g(t)t^{-\alpha/2})|}\int_{|z|>\delta\nu g(t)} u_0(z)\,dz
$$
which is small for any $\gamma,\delta>0$ fixed, if $t$ is large enough.

For $p=\infty$ we use the decay assumption $u_0\in\mathcal{D}_2$. Thus, using the bound for $F$ and integrating in radial coordinates,
\[
\begin{aligned}
{\rm IV}_1(x,t)&\le \frac C{\big|\log(g(t)t^{-\alpha/2})|}\int_{\scriptsize\begin{array}{c}|y|<\gamma|x|\\|x-y|>\delta|x|\end{array}} \frac{\big|\log(|y|t^{-\alpha/2})\big|}{|x-y|^2}\,dy\\
&\le
\frac C{\delta^2|x|^2\big|\log(g(t)t^{-\alpha/2})|}\int_{|y|<\gamma |x|} \big|\log(|y|t^{-\alpha/2})\big|\,dy\\
&= \frac {C\gamma^2}{\delta^2\big|\log(g(t)t^{-\alpha/2})|}\Big(\frac12\big|\log(\gamma |x|t^{-\alpha/2})\big|+\frac14\Big).
\end{aligned}
\]
Now, since $g(t)=o(t^{\alpha/2})$,  in the region under consideration
\[
\frac{\big|\log(\gamma|x|t^{-\alpha/2})\big|}{|\log( g(t)t^{-\alpha/2})|}\le1+\frac{|\log(\gamma\nu/2)|}{|\log(g(t)t^{-\alpha/2})|}
\le 2\quad\mbox{if }t \text{ is large}.
\]
We conclude that ${\rm IV}_1(x,t)\le C\frac{\gamma^2}{\delta^2}<\ep$
if we choose $\gamma$ small and $t$ large.

In order to bound ${\rm IV}_2$ we observe that
$$
\sup_{|\xi|\ge \gamma\nu\frac{g(t)}{t^{\alpha/2}}} F(\xi)=\max\Big\{\sup_{\gamma\nu \frac{g(t)}{t^{\alpha/2}}\le |\xi|\le 1} F(\xi),\sup_{|\xi|\ge 1} F(\xi)\Big\}\le C(1+\big|\log(\gamma\nu g(t)t^{-\alpha/2})\big|).
$$
Hence, since $g(t)=o(t^{\alpha/2})$, for $\gamma$ and $\delta$ fixed,
\[
{\rm IV}_2(x,t)\le \frac{C(1+\big|\log(\gamma\nu g(t)t^{-\alpha/2})\big|)}{|\log (g(t)t^{-\alpha/2})|}\int_{|x-y|>\delta\nu g(t)}u_0(x-y)\,dy\le C\int_{|z|>\delta\nu g(t)}u_0(z)\,dz
\]
Therefore, for all $p\in[1,\infty]$,
$$
(g(t))^{-\frac 2p}\|{\rm IV}_2(\cdot,t)\|_{L^p(\{\nu <|x|/g(t)<\mu\})}\le  C\int_{|z|>\delta c g(t)}u_0(z)\,dz<\varepsilon,
$$
for $t$ large enough, since $g(t)\to\infty$ as $t\to\infty$, and $u_0\in L^1(\mathbb{R}^2)$,
which completes the proof of~\eqref{eq-intermedio1 N=2}.

We now proceed to prove~\eqref{eq-intermedio2 N=2}. We have
\[
\begin{aligned}
\Big|\frac{t^{\a}}{|\log(g(t)t^{-\alpha/2})|}
&Z(x,t)- \kappa\Big|= \Big|\frac{F(xt^{-\alpha/2})}{|\log(g(t)t^{-\alpha/2})|}
- \kappa\Big|\le \psi(x,t)+\chi(x,t),\quad\text{where }
\\ \psi(x,t)&=\left|\frac{F(xt^{-\alpha/2})}{\big|\log(|x|t^{-\alpha/2})\big|}-\kappa\right|,\quad \chi(x,t)=\left|\frac{F(xt^{-\alpha/2})}{\big|\log(|x|t^{-\alpha/2})\big|}\right|\,\left|\frac{\big|\log(|x|t^{-\alpha/2})\big|}{\big|\log(g(t)t^{-\alpha/2})\big|}-1\right|.
\end{aligned}
\]
Remember that we are working in the region $\nu g(t)\le |x|\le \mu g(t)$. Hence, since $g(t)=o(t^{\alpha/2})$, we have that $|x|t^{-\alpha/2}\to0$. Therefore,  the behavior~\eqref{eq:constant.Nge2} yields that $\psi(\cdot,t)\to0$ uniformly as $t\to\infty$. As for $\chi$, we have
$$
1-\frac{|\log\mu|}{\big|\log(g(t)t^{-\alpha/2})\big|}\le\frac{\big|\log(|x|t^{-\alpha/2})\big|}{\big|\log(g(t)t^{-\alpha/2})\big|}\le 1+\frac{|\log\nu|}{\big|\log(g(t)t^{-\alpha/2})\big|};
$$
that is,
$$
\frac{\big|\log(|x|t^{-\alpha/2})\big|}{\big|\log(g(t)t^{-\alpha/2})\big|}=1+o(1).
$$
Therefore, since $F(\xi)/|\log|\xi||\le C$ for $|\xi|\le 1/2$, $\chi(\cdot,t)$ also goes to 0 uniformly as $t\to\infty$, and we conclude that
$$
\Big|\frac{t^{\a}}{|\log(g(t)t^{-\alpha/2})|}
Z(x,t)- \kappa\Big|\to0\quad\text{as }t\to\infty.
$$
Hence,
$$
(g(t))^{-\frac 2p}\Big\|\frac{t^{\a}}{|\log(g(t)t^{-\alpha/2})|} Z(\cdot,t)- \kappa\Big\|_{L^p(\{\nu<|x|/g(t)<\mu\})}\to 0\quad\text{as }t\to\infty,
$$
which combined with~\eqref{eq-intermedio1 N=2} yields the result.
\end{proof}

The result in the one-dimensional case follows immediately from Theorem~\ref{L^p}. As in the two-dimensional case, the behavior is given by a constant, but now it is the same one for all intermediate scales.
\begin{teo}
Let $N=1$  and $u_0\in L^1(\mathbb{R})$. For any   $g$ satisfying~\eqref{eq:intermediate.scales} and $0<\nu<\mu<\infty$,
$$
(g(t))^{-\frac 1p}\|t^{\frac\alpha2} u(\cdot,t)-MF(0)
\|_{L^p(\{\nu <|x|/g(t)<\mu\})}\to0\quad\mbox{as } t\to\infty.
$$
\end{teo}

\section{Compact sets}
\label{sect-compact sets} \setcounter{equation}{0}

In this section we determine the decay rate and the asymptotic profile in compact sets. The situation depends strongly on  the dimension.
For $N\ge 3$ the asymptotic profile will be given by the Newtonian potential of the initial datum,  a new phenonemon that is not present in the local case, with a decay rate $t^{-\alpha}$ independent of the dimension. For low dimensions we have convergence towards a constant instead, and a slower decay rate.

\subsection{The Newtonian potential} Since the large time behavior in compact sets for large dimensions will be given by the Newtonian potential of the initial datum, we start by giving conditions to guarantee that it is locally well defined.

\begin{teo}
Let $N\ge 3$ and $\mu>0$.

\noindent {\rm (a)} Let $u_0\in L^1(\R^N)$.  Then
$$
\|\Phi\|_{L^p(B_\mu)}
\le
M\|E_N\|_{L^p(B_{1})}+
M|B_\mu|^{1/p}\quad\text{for all }p\in[1,p_c).
$$

\noindent {\rm (b)} Let $u_0\in L^1(\R^N)$.  Then $\Phi\in M^{p_c}(\R^N)$.

\noindent {\rm (c)} Let $u_0\in  L^1(\R^N)\cap L^{p}_{\rm loc}(\R^N)$, $p\ge [p_c,\infty]$. Then
\begin{equation}
\label{eq:estimates.Phi}
\|\Phi\|_{L^p(B_\mu)}\le
\begin{cases}
\|u_0\|_{L^{p}(B_{\mu+1})}\|E_N\|_{L^1(B_1)}+
M|B_\mu|^{1/p},&p=p_c,\\
\|u_0\|_{L^p(B_{\mu+1})}\|E_N\|_{L^1(B_1)}+
M\|E_N\|_{ L^p(\mathbb{R}^N\setminus B_1)},\quad&p\in(p_c,\infty].
\end{cases}
\end{equation}
\end{teo}

\begin{proof} (a) The result follows easily from the decomposition $\Phi=\Phi_1+\Phi_2$, where
\begin{equation*}
\label{eq:decomposition.Phi}
\Phi_1(x)=\int_{|y|<1}\frac{u_0(x-y)}{|y|^{N-2}}\,dy,\quad \Phi_2(x)=\int_{|y|>1}\frac{u_0(x-y)}{|y|^{N-2}}\,dy.
\end{equation*}

\noindent (b) Since $E_N\in M^{p_c}(\R^N)$, the result follows immediately, because   $$
\|\Phi\|_{M^{p_c}(\mathbb{R}^N
	)}=\|u_0*E_N\|_{M^{p_c}(\mathbb{R}^N)}\le \|u_0\|_{L^1(\mathbb{R}^N)}\|E_N\|_{M^{p_c}(\mathbb{R}^N)}.
$$

\noindent (c) We perform the same decomposition as in (a). On the one hand $$
\|\Phi_1\|_{L^p(B_\mu)}\le \|u_0\|_{L^p(B_{\mu+1})}\|E_N\|_{L^1(B_1)}.
$$
On the other hand,  $\Phi_2(x)\le \int_{|y|>1}u_0(x-y)\,dy\le M$. Hence, $\|\Phi_2\|_{L^{p_c}(B_\mu)}\le M|B_\mu|^{1/p_c}$. If $p$ is supercritical we obtain a better estimate for this term, namely $\|\Phi_2\|_{L^p(B_\mu)}\le
M\|E_N\|_{ L^p(\mathbb{R}^N\setminus B_1)}$.
\end{proof}

\noindent\emph{Remark. } As a consequence of (c), if  $u_0\in  L^1(\R^N)\cap L^{p}(\R^N)$, $p\in (p_c,\infty]$, then, $\Phi\in L^p(\mathbb{R}^N)$.

\medskip

We study now the behavior at infinity of  $\Phi$, a result of independent interest which was already well known for $u_0$ integrable and compactly supported.
\begin{teo}
	Let $N\ge 3$.
	
	\noindent{\rm (a)}   If $u_0\in L^1(\R^N)\cap \D_N$, then $|x|^{N-2}\Phi(x)\to M$ as $|x|\to \infty$.
	
	\noindent{\rm (b)} Let $u_0\in L^1(\R^N)$  if $p$ is subcritical and $u_0\in L^1(\R^N)\cap \D_N$ otherwise. Then, for any $g$ going to infinity,
	$$
	(g(t))^{N\big(1-\frac 1p\big)-2}\|\Phi- ME_N\|_{L^p(\{\nu<|x|/g(t)<\mu\})}\to0\quad\mbox{as }t\to\infty, \qquad0<\nu\le \mu<\infty,
	$$
\end{teo}

\begin{proof} (a) We estimate the error as $\big||x|^{N-2}\Phi(x)-M\big|\le\textrm{I}(x)+\textrm{II}(x)+\textrm{III}(x)+\textrm{IV}(x)$, where
	$$
	\begin{aligned}
	\textrm{I}(x)&=\int_{|y|<\gamma|x|}\Big|\frac{|x|^{N-2}}{|x-y|^{N-2}}-1\Big|u_0(y)\,dy,\\
	\textrm{II}(x)&=\int_{|y|>\gamma|x|}u_0(y)\,dy,\\
	\textrm{III}(x)&=|x|^{N-2}\int_{
		\scriptsize\begin{array}{c}|y|>\gamma |x|\\ |x-y|<\delta|x|\end{array}}\frac{u_0(y)}{|x-y|^{N-2}}\,dy,\\
	\textrm{IV}(x)&=|x|^{N-2}	\int_{
		\scriptsize\begin{array}{c}|y|>\gamma |x|\\|x-y|>\delta|x|\end{array}}\frac{u_0(y)}{|x-y|^{N-2}}\,dy,
	\end{aligned}
	$$
	with $\gamma,\delta>0$  to be chosen.
	
	On the one hand, if $|y|<\gamma|x|$, with $\gamma\in(0,1)$,
	\[
	\frac1{(1+\gamma)^{N-2}}\le\frac{|x|^{N-2}}{(|x|+|y|)^{N-2}}\le \frac{|x|^{N-2}}{|x-y|^{N-2}}\le\frac{|x|^{N-2}}{(|x|-|y|)^{N-2}}\le \frac1{(1-\gamma)^{N-2}}.
	\]
	Hence, $\Big|\frac{|x|^{N-2}}{|x-y|^{N-2}}-1\Big|<\varepsilon/M$ if $\gamma$ is small enough, and hence $\textrm{I}(x)\le \varepsilon$. From now on $\gamma$ is assumed to be fixed.
	
	It is immediate that $\textrm{II}(x)<\varepsilon$ if $|x|$ is large enough, since $u_0\in L^1(\mathbb{R}^N)$.
	
	On the other hand, 	using the decay of the initial datum we get
	\[
	\begin{aligned}
	\textrm{III}(x)&\le C|x|^{N-2}\int_{
		\scriptsize\begin{array}{c}|y|>\gamma |x|\\|x-y|<\delta|x|\end{array}} \frac {dy}{|y|^N|x-y|^{N-2}}\le\frac C{\gamma^N|x|^2} \int_{{|x-y|<\delta|x|}} \frac{dy}{|x-y|^{N-2}}=C\frac{\delta^2}{\gamma^N}<\varepsilon
	\end{aligned}
	\]
	if we choose $\delta>0$ small enough.
	
	Once $\gamma$ and $\delta$ are fixed,
	\[
	\textrm{IV}(x)\le \frac1{\delta^{N-2}}\int_{|y|>\gamma|x|}u_0(y)\,dy<\varepsilon,\qquad
	\]
	if $|x|$ is large enough, using once more the integrability of $u_0$.
	
	\medskip
	
	\noindent (b) The result follows immediately from (a) if $u_0\in L^1(\R^N)\cap \D_N$. Hence, we only have to remove the decay assumption in the subcritical range. This hypothesis was only used to estimate ${\rm III}$. It is easily checked that it is enough to prove that
	$$
	(g(t))^{-\frac Np}\|{\rm III}\|_{L^p(\{\nu<|x|/g(t)<\mu\})}<\varepsilon
	$$
	if $t$ is large enough.

	By H\"older's inequality, in the region under consideration
	$$
	{\rm III}(x)
	\le \mu^{N-2}(g(t))^{N-2}\left(\int_{\scriptsize\begin{array}{c}|y|>\gamma\nu g(t)\\|x-y|<\delta \mu g(t)\end{array}} u_0(y)\,dy\right)^{\frac{p-1}p}\left(\int_{\scriptsize\begin{array}{c}|y|>\gamma\nu g(t)\\|x-y|<\delta\mu g(t)\end{array}} \frac{u_0(y)}{|x-y|^{(N-2)p}}\,dy\right)^{\frac1p}.
	$$
	Therefore, since $p$ is subcritical,
	$$
	\begin{aligned}
	(g(t))^{-\frac Np}\|{\rm III}\|_{L^p(\{\nu <|x|/g(t)<\mu\})}
	&
	\le C(g(t))^{N\big(1-\frac 1p\big)-2} \left(\int_{|z|<\delta\mu g(t)} |z|^{(2-N)p}\,dz\right)^{\frac1p}\int_{|y|>\gamma\nu g(t)} u_0(y)\,dy\\
	&\le C\int_{|y|>\gamma\nu g(t)} u_0(y)\,dy,
	\end{aligned}
	$$
	which is indeed  small  if $t$ is large enough.
\end{proof}

\subsection{Large dimensions} We begin by considering the  subcritical range.
\begin{teo}\label{teo-compacto Lp subcritico}
	Let $N\ge3$,  $\kappa>0$ as in~\eqref{eq:constant.Nge2} and $u_0\in L^1(\R^N)$. For every $p\in[1,p_c)$ and $\mu >0$
	\[
	\displaystyle\|t^\a u(\cdot,t)-\kappa\Phi\|_{L^p(B_\mu)}\to0\quad\text{as }t\to\infty.
	\]
\end{teo}
\begin{proof}
In order to  estimate the error, we take $K(t)$ going to infinity to be specified later and make the decomposition
\begin{equation*}
\label{eq:decomposition.error.compacts}
\begin{aligned}
|t^\a u(x,t)&-\kappa\Phi(x)|=
\Big|t^\a \int Z(y,t)u_0(x-y)\,dy-\kappa\int\frac{u_0(x-y)}{|y|^{N-2}}\,dy\Big|\\
\le& \int_{|x-y|<K(t)}u_0(x-y)\Big|t^{-\frac\alpha2(N-2)}F(yt^{-\alpha/2})-\frac\kappa{|y|^{N-2}}\big|\,dy\\
&+\int_{|x-y|>K(t)}u_0(x-y)t^{-\frac\alpha2(N-2)}F(yt^{-\alpha/2})\,dy+\kappa\int_{|x-y|>K(t)}\frac{u_0(x-y)}{|y|^{N-2}}\,dy.
\end{aligned}
\end{equation*}
Since $F(\xi)\le C|\xi|^{2-N}$, then $t^{-\frac\alpha2(N-2)}F(yt^{-\alpha/2})\le C|y|^{2-N}$, that  together with~\eqref{eq:estimate.error.behaviour.profile.N=3} yields
\begin{equation}
	\label{eq:estimate.A.B}
	\begin{aligned}
	|t^\a u(x,t)-\kappa\Phi(x)|&\le \mathcal{A}(x,t)+\mathcal{B}(x,t),\quad\text{where }
	\\
	\mathcal{A}(x,t)&=  C_{N,\a}\int_{|x-y|<K(t)}\frac{u_0(x-y)}{|y|^{N-2}}|y|t^{-\alpha/2}\,dy,\\
	\mathcal{B}(x,t)&=C_{N,\alpha}\int_{|y|>K(t)}\frac{u_0(y)}{|x-y|^{N-2}}\,dy,
	\end{aligned}
	\end{equation}
On the one hand,
	\[
	\begin{aligned}
	\|\mathcal{A}(\cdot,t)\|_{L^p(B_\mu)}&\le C\Big(\int_{|x|<\mu}\Big(\int_{|x-y|<K(t)}\frac{u_0(x-y)}{|y|^{N-3}}t^{-\a/2}\,dy\Big)^p\,dx\Big)^{1/p}\\
	&\le Mt^{-\a/2}\Big(\int_{|y|<2K(t)}\frac{1}{|y|^{(N-3)p}}\,dy\Big)^{1/p}\le C  Mt^{-\a/2}K(t)^{\frac Np-(N-3)}.
	\end{aligned}
	\]
We choose $K(t)$ going to infinity slowly enough, so that $t^{-\a/2}K(t)^{\frac Np-(N-3)}\to0$.
	
	On the other hand, let $t$ be large so that $\mu<K(t)/2$. Then, if $|y|>K(t)$ we have $|x-y|>K(t)/2$, so that
	\[
	\|\mathcal{B}(\cdot,t)\|_{L^p(B_\mu)}\le C_{N,\alpha} \Big(\frac2{K(t)}\Big)^{N-2}|B_\mu|^{1/p}M\to 0\quad\text{as }t\to\infty
	\]
since $K(t)$ goes to infinity.
\end{proof}

\noindent\emph{Remark. }  The choice $K(t)=t^{\frac\a2\frac p{N+p}}$ in the above proof yields the estimate (not expected to be sharp)
$$
\displaystyle\|t^\a u(\cdot,t)-\kappa\Phi\|_{L^p(B_\mu)}=O\Big(t^{-\frac{\a }2\frac{(N-2)p}{N+p}}\Big).
$$

\medskip

We now devote our attention to the critical case  $p=p_c$.  Without further assumptions on the initial data, we have the same result as in the critical range, but in the weak $M^{p_c}$ norm. As a novelty, convergence is not restricted to compact sets, and can be extended to expanding balls $ B_{g(t)}$ with $g$ growing slowly to infinity. These results are also valid in $L^{p_c}$ norm if $u_0\in   L^1(\R^N)\cap L^{p_c}_{\rm loc}(\R^N)$.

\begin{teo}\label{te0-compact-critical} Let $N\ge 3$,  $\kappa>0$ as in~\eqref{eq:constant.Nge2}, and $p=p_c$.  Let   $g$ satisfy~\eqref{eq:intermediate.scales}.
	\begin{itemize}
		\item [\rm (a)] If $u_0\in L^1(\R^N)$, then $\Phi\in M^{p_c}(\R^N)$ and,
		\[
		\big\|t^\a u(\cdot,t)-\kappa\Phi\big\|_{M^{p_c}(B_{g(t)})}\to0\quad\mbox{as } t\to\infty.
		\]
		\item [\rm (b)] If moreover $u_0\in   L^1(\R^N)\cap L^{p_c}_{\rm loc}(\R^N)$, then  $\Phi\in L^{p_c}_{\rm loc}(\R^N)$ and
		\[
		\|t^\a u(\cdot,t)-\kappa\Phi\|_{L^{p_c}(B_{g(t)})}\to0\quad\mbox{as }t\to\infty.
		\]
	\end{itemize}
\end{teo}

\begin{proof} (a)
	To prove convergence in $M^{p_c}(B_{g(t)})$, we start from the size estimate~\eqref{eq:estimate.A.B}, with $K(t)=2g(t)$.
	Since  $|y|<3g(t)$ if $|x|<g(t)$ and  $|x-y|<2g(t)$,  we get
	\[	
	\|\mathcal{A}(\cdot,t)\|_{M^{p_c}(B_{g(t)})}\le  C \|\Phi\|_{M^{p_c}(\R^N)}g(t)t^{-\a/2}<\ep\quad\mbox{if }t\ge t_\ep.
	\]
	
	On the other hand, since $|x-y|>2|x|$ if $|x-y|>2g(t)$ and $|x|<g(t)$, then
	\[
	\|\mathcal{B}(\cdot,t)\|_{M^{p_c}(B_{g(t)})}\le C \|E_N\|_{M^{p_c}(\R^N)}\int_{|y|>g(t)}u_0(y)\,dy<\ep\quad\mbox{if }t\ge t_\ep.
	\]
		
\noindent (b) To prove convergence in $L^{p_c}(B_{g(t)})$, we start from the size estimate~\eqref{eq:estimate.A.B}, with $K(t)=2g(t)$.
Since  $|y|<3g(t)$ if $|x|<g(t)$ and  $|x-y|<2g(t)$,  we get
\[
\|\mathcal{A}(\cdot,t)\|_{L^{p_c}(B_{g(t)})}\le Ct^{-\a/2}M\Big(\int_{|y|<3g(t)}\frac{dy}{|y|^{(N-3)p_c}}\Big)^{1/p_c}=
Ct^{-\a/2}M g(t).
\]

On the other hand,  if $|y|>2g(t)$  and $|x|<g(t)$, then  $|x-y|>g(t)$. Hence,
$$
\mathcal{B}(x,t)\le C(g(t))^{2-N}\int_{|y|>2g(t)}u_0(y)\,dy\quad \text{for } |x|<g(t),
$$
and therefore $
\|\mathcal{B}(\cdot,t)\|_{L^{p_c}(B_{g(t)})}\le {C}\int_{|y|>2g(t)}u_0(y)\,dy$, which ends the proof.
\end{proof}

When  $p$ is supercritical, we have to ask the initial datum to be also in $L^p_{\rm loc}(\R^N)$ in order for $\Phi$ and $u$ to be in that space, but under that assumption we have the same result as in the subcritical range.
\begin{teo}\label{teo-compacto Lp}
Let $N\ge3$,  $\kappa>0$ as in~\eqref{eq:constant.Nge2}, and $p\in(p_c,\infty]$. Given  $u_0\in L^1(\R^N)\cap L^p_{\rm loc}(\R^N)$ and $\mu>0$,
	\[
	\begin{array}{l}
	\displaystyle\|t^\a u(\cdot,t)-\kappa\Phi\|_{L^p(B_\mu)}\to0\quad\text{as }t\to\infty.
	\end{array}
	\]
\end{teo}

\begin{proof} We use the estimate~\eqref{eq:estimate.A.B} with $K$ satisfying~\eqref{eq:intermediate.scales}. For all $t$ large enough $|x|\le \mu <K(t)/2$. Hence,  $|x-y|<K(t)$ implies $|y|<3K(t)/2$. Therefore,
\[
	\mathcal{A}(x,t)\le
	C_{N,\a} K(t) t^{-\alpha/2}\Phi(x),
\]
which combined with~\eqref{eq:estimates.Phi} yields
	\[
	\|\mathcal{A}(\cdot,t)\|_{L^p( B_\mu)}\le C_{N,p,\a} K(t) t^{-\alpha/2}\big(\|u_0\|_{L^p(B_{\mu+1})}+M\big)
	\to0\quad\mbox{as }t\to\infty.
	\]
On the other hand, if $|y|>K(t)$, then  $|x-y|\ge |y|-|x|\ge K(t)/2$, and hence
	\[\begin{aligned}
	\|\mathcal{B}(\cdot,t)\|_{L^p(B_\mu)}&\le C_{N,\alpha}\int_{|y|>K(t)}u_0(y)\Big(\int_{|x|<K(t)/2}\frac{dx}{|x-y|^{(N-2)p}}\Big)^{1/p}\,dy\\
	&\le C_{N,p,\a}MK(t)^{\frac Np-(N-2)}\to0\quad\mbox{as }t\to\infty
	\end{aligned}
	\]
	since  $(N-2)p-N>0$.
\end{proof}	
	
\noindent\emph{Remarks. }  (a) The choice $K(t)=t^{\frac\a2\frac p{N+p}}$ in the above proof yields the (not necessarily sharp) estimate
$$
\displaystyle\|t^\a u(\cdot,t)-\kappa\Phi\|_{L^p(B_\mu)}=O\Big( t^{-\frac{\a }2\frac{(N-2)p-N}{(N-1)p-N}}\Big).
$$	

\noindent (b) If $u_0\in L^1(\R^N)\cap L^p(\R^N)$, $p\in(p_c,\infty]$, and $g$ satisfies~\eqref{eq:intermediate.scales}, it is easily checked, taking $K(t)=2g(t)$ in the proof of Theorem~\ref{teo-compacto Lp}, that
\[
\|t^\a u(\cdot,t)-\kappa\Phi\|_{L^{p}(B_{g(t)})}\to0\quad\mbox{as }t\to\infty.
\]
However, this result is only sharp in compact sets.

\subsection{Low dimensions}

In the critical dimension $N=2$ the  decay rate is not given by $t^{-\alpha}$; it has a logarithmic correction that makes the decay a bit slower. Moreover, the asymptotic profile is not given by the Newtonian potential of the initial datum, but by a constant. The same asymptotic constant gives the profile in expanding balls $B_{g(t)}$ as long as the decay rate is the same as in compact sets, which is the case as long as $\log g(t)/\log t\to0$.
\begin{teo}\label{teo-compactos N=2 todo p} Let $N=2$,  $\kappa>0$ as in~\eqref{eq:constant.Nge2}, and  $u_0\in L^1(\R^2)$.
	\begin{itemize}
		\item[(a)] Assume in addition that $u_0\in L^q_{\rm loc}(\R^2)$ for some $q\in (1,\infty]$ if $p=\infty$. Then, for every $p\in[1,\infty]$ and  $\mu>0$,
		\[
		\Big\|\frac{t^\a}{\log t}u(\cdot,t)-\frac{M \kappa\a}2\Big\|_{L^p(B_\mu)}\to0\quad\mbox{as } t\to\infty.
		\]
		\item[(b)] Assume in addition that $u_0\in L^q(\R^2)$ for some some $q\in (1,\infty]$ if $p=\infty$. Then, for every $p\in[1,\infty]$ and $g:\mathbb{R}_+\to\mathbb{R}_+$ such that $g(t)\to\infty$ and  $\log g(t)/\log t\to0$ as $t\to\infty$,
		\[
		(g(t))^{-\frac2p}\Big\|\frac{t^\a}{\log t}u(\cdot,t)-\frac{M \kappa\a}2\Big\|_{L^\infty(B_{g(t)})}\to0\quad\mbox{as } t\to\infty.
		\]
	\end{itemize}
\end{teo}
\begin{proof} (a) Let $\delta>0$ to be chosen. Then, for $t>1$,
	\begin{equation}
	\label{eq:decomposition.error.N=2.bounded}
	\begin{aligned}
	\Big|\frac{t^\a}{\log t} u(x,t)-&\frac{M \kappa\a}2\Big|
	\le{\rm I}(x,t)+{\rm II}(x,t)+{\rm III}(x,t),\quad\text{where }\\
	{\rm I}(x,t)&=\frac{\a}2\int_{|y|<\delta t^{\a/2}}\Big|\frac{F(yt^{-\a/2})}{|\log(|y|t^{-\a/2})|+1}-\kappa\Big| u_0(x-y)\,dy\\
	{\rm II}(x,t)&= \int_{|y|<\delta t^{\a/2}}\frac{F(yt^{-\a/2})}{|\log(|y|t^{-\a/2})|+1}\frac{\big||\log(|y|t^{-\a/2})|+1-\frac{\a}2\log t\big|}{\log t}u_0(x-y)\,dy,\\
	{\rm III}(x,t)&=\int_{|y|>\delta t^{\a/2}}\Big|\frac{t^\a}{\log t} Z(y,t)-\frac{\kappa\a}2\Big|u_0(x-y)\,dy.
	\end{aligned}
	\end{equation}

	Using the asymptotic behavior~\eqref{eq:constant.Nge2} we get ${\rm I}(x,t)\le CM\ep$
	if we choose $\delta$ small enough, so that $\|{\rm I(\cdot,t)}\|_{L^p(B_\mu)}\le C M\ep$.
	From now on $\delta$, which may be assumed to be smaller than 1, will be fixed.
	
	Let $h:\mathbb{R}_+\to \mathbb{R}_+$ be such that $h(t)\to\infty$ and $\frac{\log h(t)}{\log t}\to0$. We take $t> 1$ large, so that $\max\{1/2,\mu\}< h(t)<\delta t^{\alpha/2}/2$. After using the bound~\eqref{eq:global.bounds.F}, we  make the decomposition
	\begin{equation}
	\label{eq:decomposition.II}
	\begin{aligned}
	{\rm II}(x,t)&\le C \int_{|y|<\delta t^{\a/2}}\frac{|\log|y||+1}{\log t}u_0(x-y)\,dy= {\rm II}_1(x,t)+{\rm II}_{2}(x,t)+{\rm II}_{3}(x,t),\quad\text{where }\\
	{\rm II}_1(x,t)&= C \int_{2h(t)<|y|<\delta t^{\a/2}}\frac{|\log|y||+1}{\log t}u_0(x-y)\,dy,\\
	{\rm II}_2(x,t)&=	C \int_{1<|y|<2h(t)}\frac{|\log|y||+1}{\log t}u_0(x-y)\,dy,\\
	{\rm II}_3(x,t)&=C \int_{|y|<1}\frac{|\log|y||+1}{\log t}u_0(x-y)\,dy.
	\end{aligned}
	\end{equation}
	Since $|x|<\mu<h(t)$, then $|x-y|>h(t)$ if $|y|>2h(t)$. Hence, using also that $2h(t)>1$,
	\[
	\begin{aligned}
	{\rm II}_{1}(x,t)&\le C \int_{2h(t)<|y|<\delta t^{\a/2}}\frac{|\log(\delta t^{\a/2})|+1}{\log t}u_0(x-y)\,dy\le C\int_{|z|>h(t)}u_0(z)\,dz,\\
	{\rm II}_{2}(x,t)&\le C \int_{1<|y|<2h(t)}\frac{|\log (2h(t))|+1}{\log t}u_0(x-y)\,dy\le CM\frac{\log h(t)}{\log t},
	\end{aligned}
	\]
	so that
	\[
	\|{\rm II}_{1}(\cdot,t)\|_{L^p(B_\mu)}\le C\int_{|z|>h(t)}u_0(z)\,dz,\qquad	\|{\rm II}_{2}(\cdot,t)\|_{L^p(B_\mu)}\le CM\frac{\log h(t)}{\log t}.
	\]
	
	As for  ${\rm II}_{3}$,
	\[
	\|{\rm II}_{3}(\cdot,t)\|_{L^p(B_\mu)}\le\begin{cases}  \displaystyle\frac{CM (1+\|E_2\|_{L^{p}(B_1)})}{\log t}
	\quad &\text{if }p\in [1, \infty),\\[10pt]
	\displaystyle\frac{C\|u_0\|_{L^q(B_{\mu+1})} (1+\|E_2\|_{L^{q'}(B_1)})}{\log t}\quad &\text{if }p=\infty.
	\end{cases}
	\]

	Finally, as $F(\xi)\le C$ if $|\xi|\ge \delta$,  for all $t$ large we have
	\begin{equation}
	\label{eq:estimate.III.N=2}
	\begin{aligned}
	{\rm III}(x,t)&\le \int_{|y|>\delta t^{\a/2}}\frac{F(yt^{-\a/2})}{\log t}u_0(x-y)\,dy +
	\frac{\kappa\a}2  \int_{|y|>\delta t^{\a/2}}u_0(x-y)\,dy\\
	&\le C \int_{|y|>\delta t^{\a/2}}u_0(x-y)\,dy,
	\end{aligned}
	\end{equation}
	and hence
	\[
	\|{\rm III}(\cdot,t)\|_{L^p(B_\mu)}\le C \int_{|z|>\frac12\delta t^{\a/2}}u_0(z)\,dz.
	\]
	
	Summarizing, once we are given $\ep>0$, we first choose $\delta$ small and then  $t$ large and we get,
	\begin{align*}
	&\Big\|\frac{t^\a}{\log t} u(\cdot,t)-\frac{M \kappa\a}2\Big\|_{L^p(B_\mu)}\le C\ep.
	\end{align*}
	Since $\varepsilon>0$ is arbitrary, we get the result.
	
\noindent (b) For the second part just choose $h(t)=g(t)$ in the estimate of {\rm II}, and take into account the measure of the ball $B_{g(t)}$.
\end{proof}

The result in the one-dimensional case is a direct  corollary of Theorem~\ref{L^p}. Solutions decay more slowly than in any other dimension, and the asymptotic profile is a constant, which coincides with the one giving the asymptotic behavior in intermediate scales. In fact, the result is valid in expanding balls $B_{g(t)}$, if $g$ grows slowly to infinity.
\begin{teo}
	\label{thm:compact.N=1}
	Let $N=1$, $u_0\in L^1(\mathbb{R})$, and $p\in [1,\infty]$.
	\begin{itemize}
		\item[(a)] For any  $\mu>0$, $
	\|t^{\frac\alpha2} u(\cdot,t)-MF(0)
	\|_{L^p(B_\mu)}\to0$ as $t\to\infty$.
	\item[(b)] For any $g$ satisfying \eqref{eq:intermediate.scales}, $
	(g(t))^{-\frac1p}\|t^{\frac\alpha2} u(\cdot,t)-MF(0)
	\|_{L^p(B_\mu)}\to0$ as $t\to\infty$.
	\end{itemize}
\end{teo}

\section{Fast scales}
\label{sect-fast scales} \setcounter{equation}{0}

In this section we consider  \emph{fast scales}, satisfying~\eqref{eq:def.fast.scale}, for which the situation is more involved. As we will see, there is a difference between \emph{moderately fast} scales and \emph{very fast} scales. If the scales are moderately fast the asymptotic behavior of the solution is still given by $MZ$, with a decay rate that depends only on the scale. For very fast scales the behavior of the initial datum at infinity becomes more important, and may even determine the rate of decay and the final profile. Which scales are on each side, moderately fast or very fast, depends strongly on the decay of the initial datum at infinity.

We start by considering compactly supported initial data. We obtain convergence in uniform  relative error to $MZ$ in fast scales that are $o(t(\log t)^{\frac{2-\alpha}\alpha})$, which are hence identified as moderately fast scales.
\begin{teo}\label{teo-refined-compact support} Let $u_0\in L^1(\R^N)$ with compact support contained in the ball $B_R(0)$.
	Then, for every $\nu>0$ and $\mu\in\left(0,\left(\frac{\a(2-\alpha)}{4R\sigma}\right)^{\frac{2-\alpha}\alpha}\right)$,
	\[
	u(x,t)=M\big(1+o(1)\big)Z(x,t)\quad\mbox{uniformly in } \{\nu t^{\a/2}\le |x|\le\mu t(\log t)^{\frac{2-\alpha}\alpha}\}\quad \text{as }t\to\infty.
	\]
\end{teo}
\begin{proof}
We have,
$$
|u(x,t)-M Z(z,t)|\le t^{-\frac{\alpha N}2}\int_{|y|<R} |F((x-y)t^{-\a/2})-F(xt^{-\a/2})|u_0(y)\,dy.
$$
If  $|x|\ge \nu t^{\alpha/2}$, $|y|\le R$,  and $t\ge t_0(\nu,R)$ so that $Rt^{-\alpha/2}\le \nu/2$, then $|xt^{-\alpha/2}-syt^{-\alpha/2}|\ge \nu/2$ for all $s\in[0,1]$.
Therefore, using~\eqref{eq:estimate.MVT},
$$
\begin{aligned}
|F((x-y)&t^{-\a/2})-F(xt^{-\a/2})|\le  C \exp\Big(-\sigma\big((|x|-R)t^{-\alpha/2}\big)^{\frac 2{2-\a}}\Big)Rt^{-\a/2}\\
&\le C \exp\left(-\sigma(|x|t^{-\alpha/2})^{\frac2{2-\alpha}}\right)\exp\left(\sigma(|x|t^{-\alpha/2})^{\frac2{2-\alpha}}\big(1-(1-(R/|x|))^{\frac 2{2-\a}}\big)\right)Rt^{-\a/2}.
\end{aligned}
$$
Hence, since $1-\left(1-(R/|x|)\right)^{\frac 2{2-\a}}\le {\frac {2R}{(2-\a)|x|}} $ for all $|x|> R$, in the region under consideration we have
$$
\sigma(|x|t^{-\alpha/2})^{\frac2{2-\alpha}}\big(1-(1-(R/|x|))^{\frac 2{2-\a}}\big)\le \frac{ 2R\sigma}{2-\a}\Big(\frac{|x|}t\Big)^{\frac\alpha{2-\alpha}}\le \gamma \log t,\qquad \gamma:=\frac{2R\sigma  \mu^{\frac\alpha{2-\alpha}}}{2-\a}.
$$
Notice that, due to our assumptions on $\mu$,  $\gamma<\a/2$. Since there is a constant $C_\nu$
such that
\begin{equation}
\label{eq:bound.Z.below}
t^{-\frac{\alpha N}2}\exp\big(-\sigma(|x|t^{-\alpha/2})^{\frac2{2-\alpha}}\big)
\le C_\nu Z(x,t)\quad \text{if }|x|\ge\nu t^{\alpha/2},
\end{equation}
we conclude that
$$
|u(x,t)-M Z(z,t)|\le CMZ(x,t)t^{\gamma-\frac\a2},
$$
which yields the result.
\end{proof}

If the initial datum does not have compact support, but  has some integrable decay,  $u_0\in\mathcal{D}_\beta$ for some $\beta>N$, moderately fast scales, for which the behavior is still given by $MZ$, are not allowed to grow so fast as in the case of initial data with compact support.
		\begin{teo}\label{teo-exterior con orden} Let  $u_0\in L^1(\mathbb{R}^N)\cap \mathcal{D}_\beta$ for some $\beta>N$.  Then, for every $\nu>0$ and $\mu\in\left(0,\left(\frac\a{2\sigma}(\beta-N)\right)^{\frac{2-\alpha}2}\right)$,
		\[
		u(x,t)=M Z(x,t)(1+o(1))\quad\mbox{uniformly in } \{\nu t^{\a/2}\le |x|\le \mu t^{\a/2}(\log t)^{(2-\a)/2}\}.
		\]
\end{teo}
\begin{proof}
We split the error as
$$
\begin{aligned}
|u(x,t)-M& Z(z,t)|\le  {\rm I}(x,t)+{\rm II}(x,t)+{\rm III}(x,t),\quad \text{where}\\
{\rm I}(x,t)&=t^{-\frac{\a N}2}\int_{|y|\le \delta(t)|x|} \big|F((x-y)t^{-\a/2})-F(xt^{-\a/2})\big|u_0(y)\,dy,\\
{\rm II}(x,t)&=t^{-\frac{\a N}2}\int_{|y|\ge \delta(t)|x|} F((x-y)t^{-\a/2})u_0(y)\,dy,\\
{\rm III}(x,t)&=t^{-\frac{\a N}2}\int_{|y|\ge \delta(t)|x|} F(xt^{-\a/2})u_0(y)\,dy,
\end{aligned}
$$
Arguing as in the proof of Theorem~\ref{teo-refined-compact support} we get
\[
\begin{aligned}
{\rm I}(x,t) & \le CMt^{-\frac{\a N}2} \exp\Big(-\sigma\left((1-\delta(t))|x|t^{-\alpha/2}\right)^{\frac 2{2-\a}}\Big)|x|\delta(t)t^{-\frac\alpha2}\\
& \le CM  t^{-\frac{\a N}2}\exp\Big(-\sigma(|x|t^{-\alpha/2})^{\frac2{2-\alpha}}\Big)
\exp\Big(\sigma(|x|t^{-\alpha/2})^{\frac2{2-\alpha}}\frac{2\delta(t)}{2-\alpha}\Big)|x|\delta(t)t^{-\frac\alpha2}
\\
& \le \frac{CM}{(\log t)^{\alpha/2}}  t^{-\frac{\a N}2}\exp\Big(-\sigma(|x|t^{-\alpha/2})^{\frac2{2-\alpha}}\Big)\le o(1) MZ(x,t),
\end{aligned}
\]
if we take $\delta(t)=1/\log t$.

On the other hand, using the decay assumption on the initial datum,
$$
\begin{aligned}
{\rm II}(x,t)&\le C t^{-\frac{\a N}2}\int_{|y|>\delta(t)|x|}\frac{F((x-y)t^{-\a/2})}{|y|^\beta}\,dy\le \frac{C t^{-\frac{\a N}2}}{(\delta(t)\nu t^{\alpha/2)})^\beta}\int F((x-y)t^{-\a/2})\,dy\\
&\le C t^{-\frac{N\a}2}(\log t)^\beta t^{-\frac\a2(\beta-N)}.
\end{aligned}
$$
Since $\exp\big(-\sigma(|x|t^{-\alpha/2})^{\frac2{2-\alpha}}\big)\ge t^{-\gamma}$, $\gamma=\sigma \mu^{\frac2{2-\alpha}} $, in the region under consideration, and $\gamma<\alpha(\beta-N)/2$, we have
${\rm II}(x,t)\le CM Z(x,t) (\log t)^\beta t^{\gamma-\frac\a2(\beta-N)}=M  Z(x,t)o(1)$.

Finally,
$$
{\rm III}(x,t)\le Ct^{-\frac{\a N}2} \exp\big(-\sigma(|x|t^{-\alpha/2})^{\frac2{2-\alpha}}\big)\int_{|y|>\delta(t)\nu t^{\a/2}}u_0(y)\, dy=o(1)M  Z(x,t),
$$
and the theorem is proved.
\end{proof}

The upper restriction on the region where the behavior is given by $MZ$ in Theorem~\ref{teo-exterior con orden} is not technical, as we see next. If the initial datum has a precise (integrable) power-like decay at infinity, there are (very) fast scales for which  the large time behavior is not given by $MZ$ anymore, but by the initial datum. Convergence holds in uniform relative error.
\begin{teo}\label{teo-exterior exterior} Let  $u_0\in L^1(\mathbb{R}^N)$ be such that $|x|^\beta u_0(x)\to A$ as $|x|\to\infty$ for some $A>0$, $\beta>N$. Then, for every $\nu>\left(\frac\a{2\sigma}(\beta-N)\right)^{\frac{2-\alpha}2}$ there holds that
	\[
	u(x,t)=\frac A{|x|^\beta}(1+o(1))\quad\mbox{uniformly in }  |x|\ge  \nu t^{\frac\a2}(\log t)^{\frac{2-\a}2}.
	\]
\end{teo}
\begin{proof} Let $\gamma\in (0,(2-\a)/2)$. After adding and subtracting $A\int_{|y|>\delta|x|}Z(x-y,t)\Big(\frac{|x|}{|y|}\Big)^\beta\,dy$ for some fixed $\delta\in(0,1)$ to be chosen later, we have
	\[
	\begin{aligned}
	\big||x|^\beta u(x,t)-A|&=\Big|\int Z(x-y,t)\big(|x|^\beta u_0(y)-A\big)\,dy\Big|\le {\rm I}(x,t)+{\rm II}(x,t)+{\rm III}(x,t)+{\rm IV}(x,t),\\
	{\rm I}(x,t)&=
	\int_{|y|<\delta |x|}Z(x-y,t)\big||x|^\beta u_0(y)-A\big|\,dy,\\
	{\rm II}(x,t)&=A\int_{\scriptsize\begin{array}{c}|y|>\delta|x|\\
		|x-y|<t^{\a/2}(\log t)^\gamma\end{array}}Z(x-y,t)\Big|\Big(\frac{|x|}{|y|}\Big)^\beta-1\Big|\,dy,
	\\
	{\rm III}(x,t)&=A\int_{\scriptsize\begin{array}{c}|y|>\delta|x|\\
		|x-y|>t^{\a/2}(\log t)^\gamma\end{array}}Z(x-y,t)\Big|\Big(\frac{|x|}{|y|}\Big)^\beta-1\Big|\,dy,\\
	{\rm IV}(x,t)&=\int_{|y|>\delta |x|}Z(x-y,t)\big||y|^\beta u_0(y)-A\big|\Big(\frac{|x|}{|y|}\Big)^\beta \,dy,
		\end{aligned}
	\]

If $|y|<\delta|x|$, $\delta\in (0,1)$, then $|x-y|>(1-\delta)|x|$. Therefore,
$$
{\rm I}(x,t)\le  C t^{-\frac{\alpha N}2}\exp\Big(-\sigma(|x|t^{-\a/2})^{\frac2{2-\a}}(1-\delta)^{\frac2{2-\a}}\Big)
\int_{|y|<\delta |x|}\big||x|^\beta u_0(y)-A\big|\,dy.
$$
Since  $|x|>1$ in the region under consideration for $t$ large enough,  the integral on the right-hand side can be easily bounded by $C|x|^\beta$. Hence, using also that $(1-\delta)^{\frac\alpha{2-\alpha}}\ge 1-\frac{2\delta}{2-\alpha}$ and the bound from below for $|x|$ in terms of time,
$$
\begin{aligned}
{\rm I}(x,t)&\le C t^{\frac\alpha2(\beta-N)}\exp\Big(-\sigma(|x|t^{-\a/2})^{\frac2{2-\a}}\Big(1-\frac{2\delta}{2-\alpha}\Big)\Big)(|x|t^{-\a/2})^\beta\\
&\le C t^{\frac\alpha2(\beta-N)}\exp\Big(\big(\frac{4\delta}{2-\alpha}-1)\sigma\nu^{\frac2{2-\a}}\log t\Big)  \exp\Big(-\frac{2\delta}{2-\alpha}\sigma(|x|t^{-\a/2})^{\frac2{2-\a}}\Big)(|x|t^{-\a/2})^\beta\\
&\le C_\delta t^{\frac\alpha2(\beta-N)+\big(\frac{4\delta}{2-\alpha}-1)\sigma\nu^{\frac2{2-\a}}} \le C_\delta t^{-\frac{\sigma\theta}2},
\end{aligned}
$$
where $\theta=\nu^{\frac2{2-\a}}-\frac\a{2\sigma}(\beta-N)>0$,
if $\delta$ is small enough.

On the other hand,  if $|x-y|<t^{\a/2}(\log t)^\gamma$, then
	\[\begin{aligned}
	1-\frac{\beta}{\nu} (\log t)^{\gamma-\frac{2-\alpha}{2}}&\le \frac1{\big(1+\frac{|x-y|}{|x|}\big)^\beta}=\frac{|x|^\beta}{(|x|+|y-x|)^\beta}\le \frac{|x|^\beta}{|y|^\beta}\\
	&
	\le \frac{|x|^\beta}{(|x|-|y-x|)^\beta}=\frac1{\big(1-\frac{|x-y|}{|x|}\big)^\beta}\le 1+\frac{\beta}{\nu
	} (\log t)^{\gamma-\frac{2-\alpha}{2}}.
	\end{aligned}
	\]
		Hence,
	\[
	\Big|\frac{|x|^\beta}{|y|^\beta}-1\Big|\le  C(\log t)^{\gamma-\frac{2-\a}2},
	\]
	so that,
	\[
	{\rm II}(x,t)\le  C(\log t)^{\gamma-\frac{2-\a}2}\int F(\xi)\,d\xi\to0\quad\mbox{as } t\to\infty.
	\]

	As for ${\rm III}$, after a change of variables we get
	\[
	{\rm III}(x,t)\le C(1+\delta^{-\beta})\int_{|\xi|>(\log t)^\gamma}F(\xi)\,d\xi\to0\quad\mbox{as } t\to\infty.
	\]
	
Finally, since $|y|^\beta u_0(y)-A\to0$, $|x|^\beta<\delta^{-\beta}|y|^\beta$ and $\int F(\xi)\,d\xi=1$, ${\rm IV}(x,t)\to0$ as $t\to\infty$.
\end{proof}


\end{document}